\documentclass{article}
\usepackage[a4paper]{geometry}
\usepackage[utf8]{inputenc}

\usepackage[colorlinks]{hyperref}
\usepackage[singlelinecheck=false]{caption}
\usepackage{booktabs}

\usepackage{graphicx}
\usepackage{latexsym}
\usepackage{comment}

\usepackage{lipsum}
\usepackage{graphicx}
\usepackage{authblk}
\usepackage[english]{babel}
\usepackage{color}
\usepackage{hyperref}
\hypersetup{
    colorlinks=true,
    linkcolor=blue,
    filecolor=blue,      
    urlcolor=blue,
    citecolor=blue,
    }
\usepackage{dsfont}
\usepackage{geometry}
\usepackage{float}
\usepackage{bbm}
\usepackage{enumerate}

\usepackage[nottoc]{tocbibind}
 \usepackage{indentfirst}

\graphicspath{{./figure/}}

\usepackage{pgfplots}
\usepackage{comment}
\usepackage{fontenc}
\usepackage{booktabs}
\usepackage{tabularx}
\usepackage{tikz}
\usepackage{caption}
\usepackage{subcaption}
\usepackage{comment}
\usepackage{amsmath}
\DeclareCaptionLabelFormat{custom}
{%
      \textbf{Figure #2}
}
\DeclareCaptionLabelSeparator{custom}{ }
\DeclareCaptionFormat{custom}
{%
\centering
    #1 #2 #3 
}
 \usepackage{tikz-cd}

\usepackage{amsmath,amsfonts,amssymb,amsthm}

\numberwithin{equation}{section}
\usepackage{etoolbox}
\apptocmd{\thebibliography}{}{}{}
\theoremstyle{plain}
\newtheorem{theorem}{Theorem}[section]
\newtheorem{lemma}[theorem]{Lemma}
\newtheorem{corollary}[theorem]{Corollary}
\newtheorem{proposition}[theorem]{Proposition}

\newtheorem{reminder*}{[theorem]Reminder}

\newtheorem{details*}[theorem]{Details}

\newtheorem{comm*}{Comment}
\newtheorem{example}[theorem]{Example}
\newtheorem{definition}[theorem]{Definition} 
\newtheorem{definition*}{[theorem]Definition}

\newtheorem{notation*}{Notation}

\newtheorem{remark}[theorem]{Remark}

\title{Flux solutions for stochastic chemical systems with sources and sinks}
  \author{E. Franco, J. J. L. Velázquez}

\begin{document}

\maketitle

\begin{abstract}
    In this paper we study a class of stochastic chemical systems that,  in general,  do not satisfy the property of detailed balance nor the property of complex balance.
    These systems are obtained by adding sources and sinks to conservative chemical systems. This procedure is a way to define rigorously stochastic chemical systems in contact with reservoirs. 
    We prove that these systems are non-explosive Markov chains and we prove that they converge to a steady state as time tends to infinity. The stationary solution are out of equilibrium solutions. 
    We conclude the paper by applying our results in order to describe fluxes of molecules through some membrane channels. 
\end{abstract}

\tableofcontents

\section{Introduction} 
Chemical reaction networks can be modelled either by systems of ODEs that describe the evolution in time of the concentration of chemicals in the system or as continuous-time Markov chains (with a discrete and finite or infinite state space). The state of the Markov chain at a given time describes the number of chemicals of some given species that are present in the system. 
The deterministic models, formulated as a system of ODEs, are used to model chemical systems in which the number of molecules of each species is large. Stochastic models are  the natural way to describe chemical systems where few molecules of all, or some of the chemicals are present in the system (see for instance \cite{anderson2011continuous,anderson2015stochastic,kurtz1978strong}).

A fundamental property of systems of ODEs modelling chemical networks is the so-called property of detailed balance. 
This property holds when the system admits a positive steady state and at the steady state every chemical reaction is balanced by the corresponding reverse reaction.    
An important consequence of the property of detailed balance is the fact that when it holds, then the steady states have an explicit form, i.e. they can be written as exponentials of the Gibbs free energy. 
Moreover, in the deterministic setting the fact that detailed balance holds can be used to define a Lyapunov functional that allows to study the stability of the steady states.
This Lyapunov functional decreases along solutions. Specifically, this functional is  the free energy, that decreases in closed isothermal chemical systems.

Similarly, a continuous-time Markov chain that models a chemical system satisfies the property of detailed balance when it is reversible, i.e. when there is a stationary distribution at which the chemical reactions are individually balanced. We refer to \cite{anderson2010product,liggett1985interacting} and to Section \ref{sec:lack of detailed balance when sources and sinks} for more details on this. 

The detailed balance property has important mathematical consequences and an important physical meaning. Indeed, the physical principle of microscopic reversibility states that a chemical system that is at constant temperature and that does not exchange energy and chemicals with the environment must satisfy the property of detailed balance (see for instance \cite{avanzini2024methods,avanzini2025coarse,boyd1974detailed,polettini2014irreversible}). 
However biochemical systems often exchange energy and chemicals with the surroundings. As a consequence, they are usually modelled by systems of equations that do not satisfy the detailed balance property or by continuous-time Markov chains that are non-reversible. 

In \cite{franco2025detailed} we proved that a way to obtain chemical systems in which the detailed balance fails is to freeze at constant values the concentration of some of the chemicals in a system that satisfies the property of detailed balance. More precisely, suppose that a given system of ODEs, modelling a chemical network satisfies the detailed balance property.
Assume now that  the concentration of some of the chemicals in the system does not evolve according to the dynamics induced by the system of ODEs, but are instead frozen at constant values, for instance because the system is in contact with a reservoirs of chemicals. 
If the reduced chemical system of ODEs for the non-frozen concentrations has more cycles than the original system which satisfies the detailed balance property, then this reduced chemical system satisfies the detailed balance property if and only if the frozen concentrations are chosen at equilibrium values, i.e. are fine tuned. The onset of cycles induced by the freezing of some concentrations of chemicals in a network is a known phenomena, indeed these cycles are called emergent cycles in the physical literature \cite{polettini2014irreversible}. 
Notice that the fact that the detailed balance property imposes conditions on the chemical rates of reactions that belong to the cycles is also well known (see \cite{wegscheider1902simultane}). 

One of the question that we address in this paper is how to model a stochastic chemical system that is in contact with a reservoir of chemicals. In order to do that we add sources and sinks to a stochastic chemical system that satisfies the detailed balance property and that is conservative. 
Adding sources and sinks to a chemical system can introduce new cycles and hence can break the property of detailed balance. 
Notice that to add sources and sinks to a stochastic chemical system is not the same as assuming that the chemical system is in contact with reservoirs with a given chemical potential. 
This, indeed, would implicitly assume that the probability  distribution can be approximated by a macrocanonical distribution.

From the mathematical point of view, the analysis of chemical systems that do not satisfy the property of detailed balance is challenging as few techniques are available to study their dynamics.
One of the difficulties  in the analysis of systems of ODEs that model a chemical network that does not satisfy the property of detailed balance is to prove that the system admits a unique solution that is valid for all times. Due to the non-linearity of the chemical reactions it is possible to have systems of ODEs, modelling chemical systems, whose solution escape to infinity in finite time. The fact that solutions to a system of ODEs modelling chemical systems and that satisfy the property of detailed balance or the so called property of complex balance do not escape to infinity in finite time has also been proven in \cite{feinberg2019foundations}. 
The existence of time dependent solutions that are bounded in finite times has been proved in \cite{anderson2011boundedness} also for 
weakly reversible networks with one linkage class.

In the stochastic setting we can have a similar situation, namely the Markov chain can be explosive, i.e. an infinite number of jumps can occur in finite time.
In order to exclude this pathological behaviour, assumptions on the chemical reaction rates are needed. This problem has been analysed for chemical system that satisfy the so-called complex balance property. 
This property is a generalization of the property of detailed balance. 
It has been proven in \cite{anderson2018non} that continuous-time 
Markov chains that model chemical systems that satisfy the property of complex balance are non-explosive, i.e. the probability of having a finite number of molecules is equal to $1$ for all times.

Another aspect that is relevant is the analysis of the long-time dynamics of chemical systems. 
The existence and  the stability of stationary solutions to systems of ODEs modelling chemical systems that have deficiency zero has been analysed by several authors (see for instance \cite{feinberg1972complex,feinberg2019foundations,horn1972general}). 
These results have been extended also to the stochastic setting in which the chemical system is  modelled by a continuous-time Markov process
(see for instance \cite{anderson2010product,cappelletti2016product}). 

In this paper we study  a class of continuous-time Markov chains that model chemical systems. 
We consider a specific class of stochastic chemical systems that do not satisfy neither the property of detailed balance or the property of complex balance.
The class of systems that we study in this paper are chemical systems that exchange chemicals with the environment only via sources and sinks of chemicals, i.e. reactions of the form 
\[ 
\emptyset \quad \rightleftarrows \quad (k), \quad  \text{ where } k \in \Omega
\]
where $ \Omega  $ is the finite set of the species in the network. 
The chemical reactions that are not sources and sinks are assumed to be conservative (we refer to Definition \ref{def:conservative} for the precise definition of conservative reaction).
We call this class of chemical networks, obtained by adding sources and sinks to conservative chemical systems, \textit{conservative systems with sources and sinks}. 
In particular we assume that the sources and sinks satisfy the mass action law, however it would be possible to model sources and sinks in a more complicated way, than the one considered in this paper. For instance assuming that the rate at which molecules are injected in the system depends on the concentration  of some  molecules present in the system.

Notice that adding sources and sinks to a stochastic chemical system is a precise mathematical way to define a chemostat (or a reservoir of chemicals) for stochastic chemical systems. We refer to \cite{rao2018conservation,remlein2025chemostat} for related results in the physical literature. 
The first result that we prove in this paper is that if we add sources and sinks to a conservative stochastic chemical system that satisfies the detailed balance property, then we obtain a system that does not satisfy the detailed balance property,  unless either the chemical rates are chosen in a very specific manner, or the sources and sinks and the chemical reactions satisfy a particular topological property. 
We stress that to add sources and sinks to a  stochastic chemical system with conservative chemical reactions that satisfy the detailed balance condition is the natural way to model open stochastic chemical systems. 

We then introduce the master equation associated with the class of continuous-time Markov chains that model conservative chemical systems with sources and sinks
\begin{align} \label{eq:master with sources intro}
    \partial_t f (t, n) &= \sum_{ \rho } [R_\rho(n - S_\rho ) f(t,n-S_\rho ) + R_{-\rho} ( n + S_\rho ) f(t, n+ S_\rho ) ] \\
    & -f(t, n)  \sum_{ \rho }(R_\rho (n) + R_{- \rho} (n) ) + \sum_{ k  } [ A_k(n - e_k )   f(t,  n - e_k ) + B_k(n+e_k) f(t, n+e_k) ] \nonumber \\
    &-  \sum_{ k} [ A_k(n)  + B_k(n)  ] f( t, n  ). \nonumber 
\end{align}
Here $ R_\rho (n) $ is the rate at which the chemical reaction $ \rho $ takes place if the state of the system is $n \in \mathbb N^{ |\Omega | } $. The stoichiometric vector associated with the chemical reaction $ \rho $ is the vector $ S_\rho \in \mathbb Z^{ |\Omega |} $. Similarly,  $ R_{ - \rho} (n) $ is the rate at which the reverse of the chemical reaction $ \rho $ occurs. The reverse of the chemical reaction $S_\rho$ is the chemical reaction associated with stoichiometric vector $ - S_\rho$. 
Similarly, $A_k (n) $ is the rate at which chemicals of type $k \in \Omega  $ are injected in the system by  the sources and $B_k (n) $ is the rate at which chemicals are removed by the sink. The stoichiometric vector corresponding to the source is $ e_k $ and the one corresponding to the sink is $ -e_k $. 
We assume that the chemical rates of the sources and sinks satisfy the mass action law, the detailed definition of the rates are given in \eqref{mass action source}-\eqref{mass action sink}. 
In this paper we assume that the set of the sinks and sources is finite, as well as the set of chemical reactions.

The solution to the master equation \eqref{eq:master with sources intro} describes the evolution in time of the probability $f(t, n ) $ that the system is at a certain state $ n \in \mathbb N^{|\Omega |} $ at a given time $t\geq 0 $. Here $ \Omega $ is the set of the species in the system and $n =(n_k)_{ k \in \Omega } $ is the vector that encodes the number of molecules of each chemical $k \in \Omega $ in the system. 

In this paper we analyse the solution to equation \eqref{eq:master with sources intro} using a PDE approach. It would be possible to study the same problem using a probabilistic approach and using the theory of Markov processes (see for instance \cite{liggett1985interacting}). Notice however that the Markov chains that we are considering have an infinite state space, hence it is not possible to use the classical theory of finite Markov chains. 

We prove the existence of a time dependent solution $f(t, n ) $ to the master equation that models a conservative chemical system with sources and sinks.
In particular, we prove that the solution is a probability for all times, hence $ \sum_{ n } f(t, n ) =1 $ for every $t \geq 0 $. This means that the corresponding continuous-time Markov chain is non-explosive (see \cite{norris1998markov}). 
This result holds both for mass action and non-mass action kinetics. 
In order to prove this it is crucial to observe that the fact that the chemical reactions are conservative, together with the form of sources and sinks, implies that if the initial condition $f_0 $ is such that $\sum_{n \in \mathbb N^{ |\Omega |}}  |n | f_0( n ) < \infty  $, then there exists a positive constant $C>0$ that is such that 
\begin{equation} \label{bound on mass intro}
\sum_{n \in \mathbb N^{ |\Omega |}} |n | f(t, n ) \leq C ,\quad \forall t \geq 0. 
\end{equation}
Notice that in order to prove the existence of a time dependent solution $f (t, n) $ to the master equation \eqref{eq:master with sources intro} we do not need to assume that the chemical reactions satisfy the detailed balance condition. 
However, even if the result holds also when the detailed balance property fails, we stress that the chemical systems in which we are interested are chemical systems obtained by adding sources and sinks to chemical systems that satisfy the detailed balance condition and that are conservative. These are the open chemical systems with lack of detailed balance that are meaningful from the physical point of view. 

We will prove that when the kinetics of the chemical system are given according to the mass action law then the time dependent solution to the master equation \eqref{eq:master with sources intro} is unique. Moreover, under the mass action assumption, we can prove that the master equation \eqref{eq:master with sources intro} corresponding to a conservative system with sources and sinks has a unique steady state $ \overline f $.
In order to prove the existence of a steady state we use the bound \eqref{bound on mass intro} as well as fixed point arguments. To prove the uniqueness of the stationary solution $ \overline f $ we study the dual to equation \eqref{eq:master with sources intro}. 
In particular, we prove that the solution to the adjoint problem converges to a constant as time tends to infinity. 
We use this to prove the uniqueness of the stationary solution $ \overline f $ as well as the fact that the time dependent solution $f$ converges to the steady state $ \overline f $ as time tends to infinity. 

 In the last part of the paper we consider applications of the results of this paper to two different models of membrane channels (see \cite{alberts2022molecular,catacuzzeno2024building,mackey2013ion}). 
Ion channels connect the interior of a membrane with the exterior of the membrane. They are located across the membrane and, when open, they allow some molecules to pass from one side of the membrane to the other.

The first model of membrane channels that we consider is a model in which the transport of molecules is driven by the difference of concentration of molecules inside and outside the membrane. 
For this model of membrane channels  we assume that only one type of molecule (or a ion)  is present outside and inside the membrane. The channel can be open towards the interior of the membrane or towards the exterior of the membrane. In particular, the transitions between different states of the channel occurs randomly. When the channel is open toward the outside of the membrane the molecules that are outside the membrane can enter the channel and then enter the membrane when the channel opens toward the interior. In a similar manner a molecule can go from the inside of the membrane  to the outside using the channel. We will provide more details on the model in Section \ref{sec:example 1}. 
If the number of molecules outside the membrane is higher in the exterior of the membrane than in the interior, then there is a net transport of molecules from the outside to the inside of the membrane, i.e. the transport of molecules is induced by the gradient of concentrations.  

We model the interaction between the channels and the molecules as chemical reactions that are conservative. Moreover we assume that there is a source of molecules both inside and outside the membrane. This allows to prove the existence of a non-equilibrium stationary solution for this model. 
The stationary solution is out of equilibrium because at steady state there are fluxes of molecules through the channels. 

The second model that we consider is an active transport model of channels in which molecules are  transported against their gradient of concentration using the energy of a gradient of ions. 
Therefore in this second model we assume that two types of chemicals are co-transported through the channels. We assume that the channel can switch from being open towards the inside to being open towards the outside (or vice versa from being open towards the inside to being open towards the outside) only when two different types of molecules are both present in the channel. 
This type of transport through the channel is called co-transport in the biological literature. 
This mechanism is used, for instance, to transport glucose into the cell using the electrochemical gradient of Na$^+$. Indeed the concentration of  Na$^+$ is abundant in the extracellular space while it is low in the cytosol. 
On the contrary, the concentration of glucose is high in the cytosol and low in the extracellular space. 
The channel can switch state only when a ion Na$^+$ and a molecule of glucose are at the same time in the channel. Since the concentration of Na$^+$  is higher in the extracellular space this allows to transport glucose inside the cell against its gradient of concentration (i.e. even if the concentration of glucose is higher inside the cell than outside). 

In this paper we write the master equation of a model of co-transport and we analyse its long-time behaviour. In particular we prove the existence of a unique stationary solution that is attractive. 
We stress that this stationary solution is a non-equilibrium stationary solution that cannot be written in terms of the Gibbs free energy.

\subsection{Notation} 
We introduce now the notation that will be used in the paper.  
We denote with $e_k $ the $k$-th vector of the canonical basis of $\mathbb R^n $. 
We denote with $ \mathbb N $ the set of the non-negative natural numbers including the zero. 
We use the notation $n ! = \prod_{k =1}^M (n_k) ! $ for $n \in \mathbb N^M$ for $M \geq 1 $ and the notation $ |n | = \sum_{ k \in \Omega  } n(k)  $. 
Finally we use the following notation for the sequence $( f)_{ n \in \mathbb N^M}  $ that are probabilities in $ \mathbb N^M $
\[
\mathcal P ( \mathbb N^M ) = \left\{ (f_n )_{ n \in  \mathbb N^M } : f_n \geq 0 , \sum_{n \in  \mathbb N^M  } f_n = 1  \right\}. 
\]
The set $c_{00}(\mathbb N^M ) $ is the set of the sequences in $\mathbb N^M $ that have compact support. 
In the paper it will also be convenient to denote with $\ell^1_{\text{weak}} (\mathbb N^M ) $ the space $ \ell^1 (\mathbb N^M )$ endowed with the $\ast-$weak topology.

Assume that $ U \subset \Omega $ and that $ v \in \mathbb R^\Omega  $ we denote with $\pi_U v \in \mathbb R^{U} $ the following vector 
\begin{equation} \label{proj reactions}
\pi_U v (k) =v(k ) , \quad k \in U. 
\end{equation}
\subsection{Plan of the paper} 
In Section \ref{sec:chemical systems} we introduce the class of models that we study in this paper. 
In Section \ref{sec:main} we summarize the main results that we prove in this paper. 
In Section \ref{sec:lack of detailed balance when sources and sinks} we show that the chemical systems that we consider in general do not satisfy the property of detailed balance.
More precisely we prove that the chemical system obtained adding sources and sinks to a chemical system that satisfies the property of detailed balance, does not satisfy the detailed balance property unless the sources and sinks rates are fine tuned or they satisfy a topological property.
In Section \ref{sec:existence} we prove the existence of a unique time dependent solution to equation \eqref{eq:master with sources intro}. In Section \ref{sec:convergence} we prove that the time dependent solution converges to a steady state as time tends to infinity. 
Finally in Section \ref{sec:examples} we apply the results of the paper to some examples of fluxes of molecules through channels in membranes. 

\section{Chemical systems with sources and sinks and conservative reactions} \label{sec:chemical systems}
In this paper we study a class of chemical systems in which we model two main types of reactions, namely conservative reactions and reactions that are sources and sinks. 
In this section we define this class of chemical systems and the corresponding master equation.

\subsection{Chemical systems} 

We assume that $\Omega := \{ 1, \dots, M  \} $, with $M \geq 1 $ is the set of the substances present in the system. 
A chemical system is a set of $r>0$ reactions between the substances in $\Omega $ and a set of sources and sinks. 
Chemical reactions map a non empty set of reactants into a non empty set of products, i.e. 
\[ 
\sum_{k \in \Omega } \nu_\rho (k)  \quad \overset{\rho}{\underset{-\rho}{\rightleftarrows}} \quad  \sum_{k \in \Omega } \nu_{-\rho}  (k) , \quad  \rho \in \{ 1 , \dots, r \}
\]
where $ \nu_\rho (k)$ is the number of molecules of substance of type $k \in \Omega $ that are reactants for the reaction $\rho $ while $ \nu_{-\rho} (k)$ is the number of molecules of substance of type $k \in \Omega  $ that are products for the reaction $\rho $. 
A source/sink $e_k, - e_k $, instead is a reaction of the form 
\[ 
\emptyset \quad \overset{e_k}{\underset{-e_k}{\rightleftarrows}} \quad (k), \quad k \in \Omega_c \subset \Omega. 
\]

More precisely, a \textit{chemical reaction} $ S_\rho$ between elements of $\Omega $ is a vector in $\mathbb Z^M$ with mixed signs, i.e. such that there exist $j, k  \in \Omega $ with $S_\rho (k) >0 $ and $S_\rho (j ) <0 $. 
 We denote with $\nu_\rho \in \mathbb N^M $ the vector defined as 
\begin{align*}
\nu_\rho( k )&:=  - S_\rho (k) \quad  \text{ if }  S_\rho(k) < 0  \\ 
\nu_\rho( k )&:= 0  \quad \text{ if } S_\rho(k) \geq  0.
\end{align*} 
Notice that by assumption we have that for every chemical reaction $S_\rho $ it holds that $\nu_\rho \neq 0 $ and $S_\rho - \nu_{\rho } \neq 0 $. 
We denote with $S_{- \rho } $ the reverse of the reaction $S_\rho$, i.e. $S_{- \rho} = - S_\rho$. 
In this paper we assume that the network is bidirectional, i.e. if the chemical reaction $S_\rho $ belongs to the system, then also the reverse reaction $S_{- \rho} $ belongs to the system. 
Moreover we assume that the set of the reactants and the set of the products of a chemical reactions are disjoint sets, i.e. we assume $\nu_\rho^T \nu_{- \rho} =0  $.
This assumption excludes autocatalytic reactions of the form $(A)+ (B) \rightarrow (A) $.

We define the stoichiometric matrix $\mathbb S_r $ of the chemical reactions as the matrix that has as columns the reaction vectors $\{ S_\rho \}_{\rho=1}^r $, i.e. 
\begin{equation} \label{stoic matrix}
\mathbb S_r := (S_\rho )_{ \rho =1}^r  =\left( \begin{matrix}
    S_1  S_2 \dots S_r
\end{matrix} \right) . 
\end{equation}
A \textit{source} is a vector of the form $ e_k \in \mathbb R^M  $ for some $k \in \Omega $. Similarly, a \textit{sink} is a vector of the form  $ - e_k \in \mathbb R^M $ for some $k \in \Omega $.  We assume that every time there exists a source of a chemical of type $ k \in \Omega $ there exists also a sink of the same chemical $k \in \Omega $. The set of the chemicals that are added and removed from the system with sources and sinks is the set $\Omega_c \subset \Omega$. 

The stoichiometric matrix associated to the chemical reactions and the sources and sinks is the following matrix in $\mathbb R^{M \times (r + |\Omega_c| ) } $
\begin{equation} \label{stoic matrix ext}
\mathbb S := \begin{bmatrix}
    \mathbb S \vert \textbf {I}_{M \times |\Omega_c|} 
\end{bmatrix}
= \begin{bmatrix}
     S_1 S_2 \dots S_r  \vert e_1 e_2 \dots e_{|\Omega_c|}
\end{bmatrix}. 
\end{equation}

Let us now specify the assumptions that we make on the \textit{kinetics} of the chemical reactions $ S_\rho$. 
In particular a kinetic associated with the chemical reaction $S_\rho $ is a map $ 
R_\rho : \mathbb N^M \rightarrow \mathbb R_+ $ that is such that $R_\rho (n) \neq 0$ if and only if $ (n - \nu_\rho )\in  \mathbb N^M$. 
Most of the results that we prove in this paper hold for chemical systems that are mass action systems, i.e. the  rate $R_\rho(n) $ at which the chemical reaction $S_\rho$ occurs when the system is at state $n \in \mathbb N^M $ is given by
\begin{equation} \label{mass action reactions} 
R_\rho (n) =\begin{cases} 
 K_\rho  \frac{ n!  }{ (n - \nu_\rho )!}, & \text{ if }  (n - \nu_\rho )\in \mathbb N^M,  \\
  0 &\text{ if } (n - \nu_\rho )\in \mathbb Z^M \setminus \mathbb N^M
\end{cases} 
\end{equation}
where $K_\rho >0 $ is the chemical rate associated with the reaction $S_\rho$. 
Similarly, we assume that a substance $k \in \Omega_c $ is created by a source at a rate $A_k (n)$  that is independent on the state $n \in \mathbb N^M $ of the system, i.e. we assume that for every $ k \in \Omega_c$
\begin{equation} \label{mass action source}
A_k (n) :=\alpha_k \quad \forall n \in \mathbb N^M 
\end{equation}
where $\alpha_k >0 $. 
Instead, the rate $B_k(n) $ at which a chemical $k \in \Omega_c $ is removed from the system by a sink depends linearly on the state $n \in \mathbb N^M $. Indeed we assume that for every $k \in \Omega_c $
\begin{equation} \label{mass action sink}
B_k  (n) :=\begin{cases} 
 \beta_k n _k, & \text{ if }  (n - e_k )\in \mathbb N^M,  \\
  0 &\text{ if } (n - e_k )\in \mathbb Z^M \setminus \mathbb N^M 
\end{cases} 
\end{equation}
with $ \beta_k >0 $.
In this paper we only consider sources and sinks that satisfy the mass action law \eqref{mass action source}-\eqref{mass action sink}.

\begin{definition}[Chemical system] \label{def:chemical}
Assume that $\mathcal R$ is a set of bidirectional chemical reactions between elements of $\Omega $, i.e. 
 $\mathcal R := \{ S_\rho \}_{\rho=1}^r \cup  \{ S_{- \rho }\}_{\rho=1}^r $ where $S_{\rho_1} \neq - S_{\rho_2 } $ for every $\rho_1 \neq \rho_2 $. Assume that $ \mathcal R_s $ is a set of sources and sinks, i.e.  $\mathcal R_s := \{e_k \}_{k \in \Omega_c}  \cup  \{ - e_{ k  }\}_{\Omega_c }  $ with $ \Omega_c \subset \Omega$.  
 Assume that $\mathcal K := \{ R_\rho  \}_{ \rho =1 }^r \cup \{ R_{-\rho}\}_{\rho=1}^r  $ is the set of kinetics associated with the chemical reactions $S_\rho, S_{- \rho } \in \mathcal R $ and $\mathcal K_s := \{ \alpha_k  \}_{ k \in \Omega_c }  \cup \{ \beta_k  \}_{ k \in \Omega_c } $ is the set of the chemical rates associated with the sources and sinks. 
 We say that  $(\Omega, \mathcal R,  \mathcal R_s , \mathcal K, \mathcal K_s )$ is a chemical system. Moreover 
 \begin{itemize}
\item if $ \Omega_c \neq \emptyset $, then we say that $(\Omega, \mathcal R,  \mathcal R_s , \mathcal K, \mathcal K_s )$ is a chemical system with sources and sinks. 
\item If $ \Omega_c = \emptyset $, then we say that $(\Omega, \mathcal R,  \mathcal R_s , \mathcal K, \mathcal K_s )$ is a chemical system without sources and sinks. We denote it with $(\Omega, \mathcal R,  \mathcal K)$. 
\end{itemize}
\end{definition}
Notice that  in Definition \ref{def:chemical} we are assuming that a chemical system contains only bidirectional reactions. 

In this paper we say that a chemical system is a \textit{mass action chemical system} if every kinetics is a mass action kinetics, i.e. is given by \eqref{mass action reactions}. 

We define the set of the conserved quantities associated with a set of chemical reactions $\mathcal R = \{ S_\rho\}_{ \rho =1}^r \cup   \{ S_{-\rho} \}_{ \rho =1}^r$ as 
  \begin{equation} \label{eq:stochio}
\mathcal M  :=  {\operatorname{span}\{ S_\rho : \rho \in \{ 1, \dots, r\}   \}}^{\perp}.  
    \end{equation}
We define now the set of physically relevant non-negative conservation laws
$\mathcal M_+:=\mathcal M \cap \mathbb R_*^N .$
A conservative set of reactions is a set of reactions that is such that every substance $ k \in \Omega $ appears in a conserved quantity $ m \in \mathcal M_+$, i.e. for every $k \in \Omega $ there exists a  $ m \in \mathcal M_+$ such that $m (k ) >0 $. More precisely we have the following definition. 
\begin{definition} [Conservative set of reactions] \label{def:conservative}
Assume that $\mathcal R $ is a set of chemical reactions. Then we say that $\mathcal R $ is conservative if
\begin{equation}\label{eq:conservative}
\mathcal M_+ \cap \mathbb R_+^M \neq \emptyset.
\end{equation}
\end{definition}
We now define a suitable basis for the space $\mathcal M$.
Notice that the set of the positive conserved quantities $\mathcal M_+ $ is a cone, it turns out that it is finitely generated by its extreme rays, i.e. the set of the $m \in \mathcal M_+$ that cannot be written as $m = \gamma_1 m_1 + \gamma_2 m_2 $ for  $m_1, m_2 \in \mathcal M_+$ linearly independent and $\gamma_1, \gamma_2 >0$, see \cite{rockafellar1997convex}. 

\begin{lemma} \label{lem:positivity of the conservation laws}
If a set $ \mathcal R$ of chemical reactions is conservative, then the set of the extreme rays $\mathcal B= \{ m_k \}_{k =1 }^\ell $ of the positive cone $\mathcal M_+$ are a basis of $\mathcal M $. 
\end{lemma}
We call the set $\mathcal B $ the \textit{extremal basis} of the conserved quantities. We refer to \cite{franco2025detailed} for the proof of Lemma \ref{lem:positivity of the conservation laws}. 

Since in this paper we are interested in analysing the property of detailed balance of chemical systems we introduce the definition of cycles. 
\begin{definition}[Cycles] \label{def:cycles}
The space of the cycles of a chemical system  $(\Omega, \mathcal R, \mathcal R_s,  \mathcal K, \mathcal K_s )$ is defined as
\begin{equation*}
\mathcal W := \left\{  w \in \mathbb R^{r + |\Omega_c|} : \sum_{\rho=1}^{r } w(\rho) S_\rho + \sum_{ k \in \Omega_c } w(k ) =0 \right\}. 
\end{equation*}
\end{definition}

\subsection{The master equation}

The master equation describing the evolution in time of a stochastic chemical system with sources and sinks  $(\Omega, \mathcal R,  \mathcal R_s , \mathcal K, \mathcal K_s )$  is the following equation
\begin{align} \label{eq:master with sources}
    \partial_t f (t, n) &= \sum_{ \rho =1}^r [R_\rho(n - S_\rho ) f(t,n-S_\rho ) + R_{-\rho} ( n + S_\rho ) f(t, n+ S_\rho ) ] \\
    & -f(t, n)  \sum_{ \rho =1}^r(R_\rho (n) + R_{- \rho} (n) ) + \sum_{ k \in \Omega_c } [ A_k(n - e_k )   f(t,  n - e_k ) + B_k(n+e_k) f(t, n+e_k) ] \nonumber \\
    &-  \sum_{ k \in \Omega_c } [ A_k(n)  + B_k(n)  ] f( t, n  ). \nonumber 
\end{align}
Here $f(t, n ) $ represents the probability that at time $t >0$ the system has state $n \in \mathbb N^M $. Therefore we will look for solutions that are such that $f(t, \cdot ) \in \mathcal P (\mathbb N^M)$ for every $t \geq 0 $.

A steady state to the master equation \eqref{eq:master with sources} is a $ g \in \mathcal P (\mathbb N^M ) $ that is such that 
\begin{align} \label{eq:master with sources_stst}
    0 &= \sum_{ \rho =1}^r [R_\rho(n - S_\rho ) g(n- S_\rho ) + R_{-\rho} ( n + S_\rho ) g( n+ S_\rho ) ]  -g( n)  \sum_{ \rho =1}^r(R_\rho (n) + R_{- \rho} (n) ) \\
    & + \sum_{ k \in \Omega_c } [ A_k(n - e_k )  g( n - e_k ) + B_k(n+e_k) g(n+e_k) ] -  \sum_{ k \in \Omega_c } [ A_k(n)  + B_k(n)  ] g( n  ). \nonumber 
\end{align}

It is convenient to formulate the master equation in weak form. To this end we multiply the equation \eqref{eq:master with sources} by a test function $ \varphi \in c_{ 00} (\mathbb N^M) $ and we sum over all  $n \in \mathbb N^M $
\begin{align} \label{eq:master with sources weak}
\sum_{  n \in \mathbb  N^M  } \varphi(n) \partial_t   f (t, n) = \sum_{ n \in \mathbb N ^M }  f (t, n)  \mathcal L_s [\varphi] ( n ) +  \sum_{ n \in \mathbb N ^M }   f (t, n)  \mathcal L_r   [\varphi] (n ) 
\end{align}
where we use the notation 
\begin{align*}
\mathcal L_r   [\varphi] (n )  : =  \sum_{ \rho =1}^r  R_\rho(n) \left[ \varphi(n+ S_\rho) - \varphi(n)  \right] +   \sum_{ \rho =1}^r   R_{-\rho}(n) \left[ \varphi(n- S_\rho) - \varphi(n)  \right]  \nonumber \\
\end{align*}
\begin{align*} 
     \mathcal L_s  [\varphi] (n )  :=       \sum_{ k \in \Omega_c }  A_k (n)   \left[ \varphi(n+ e_k ) - \varphi(n)  \right] +    \sum_{ k \in \Omega_c } B_k (n)   \left[ \varphi(n-  e_k ) - \varphi(n)  \right].  \nonumber
\end{align*}

\section{Main results} \label{sec:main}
In this section we state the main results that we prove in this paper for the class of systems introduced in Section \ref{sec:chemical systems}.

As a first step we prove that the chemical systems that are obtained adding sources and sinks to a stochastic chemical system that satisfies the detailed balance property, do not satisfy the detailed balance property unless the sources and sinks do not belong to any cycle or the sources and sinks rates are selected to be at equilibrium values. 

We start introducing the definition of detailed balance. 
\begin{definition}[Detailed balance]
A chemical system  $(\Omega, \mathcal R,  \mathcal R_s , \mathcal K, \mathcal K_s )$ satisfies the property of detailed balance if there exists a  $f_s  \in \mathcal P (\mathbb N^M ) $ that is  such that 
\begin{equation} \label{eq:detailed balance reactions}
R_\rho (n) f_s( n)  = R_{- \rho}(n+ S_\rho) f_s( n + S_\rho), \quad \forall n \in \mathbb N^M, \  \forall \rho \in \{ 1, \dots, r \}
\end{equation}
and 
\begin{equation} \label{eq:detailed balance reactions}
A_k (n) f_s( n)  = B_k (n+ e_k) f_s( n + e_k ), \quad \forall n \in \mathbb N^M, \  \forall k \in \Omega_c. 
\end{equation}
\end{definition}
Notice that by its definition $f_s $ is a steady state to the master equation \eqref{eq:master with sources}.

  \begin{theorem} \label{cor:sources break detailed balance} 
  Assume that $(\Omega, \mathcal R , \mathcal R_s ,  \mathcal K, \mathcal K_s )$ is a mass action chemical system with sources and sinks.  Let $\mathcal W_r $ be the set of the cycles associated with $(\Omega, \mathcal R , \emptyset ,  \mathcal K, \emptyset )$ and  $ \mathcal W$  the set of the cycles associated with $(\Omega, \mathcal R , \mathcal R_s ,  \mathcal K, \mathcal K_s )$. 
  \begin {enumerate} 
 \item  Assume that
\begin{equation} \label{eq for the parameters}
\log \left( \frac{\alpha_k }{\beta_k }\right) =E(k), \quad \forall k \in \Omega_c^{\mathcal W} 
\end{equation}
where $E \in \mathbb R^M $ is the  solution to 
 \begin{equation} \label{energy reactions}
 S_\rho^T  E = \ln \left( \dfrac{ K_{ - \rho} }{ K_{  \rho } }  \right)  \ \forall \rho \in \{ 1, \dots, r \}
\end{equation} 
and where 
\begin{equation}
    \Omega_c^{\mathcal W }:= \{ k \in \Omega : \exists w \in \mathcal W  \quad (\mathbb S_r  \pi_R w ) ( k ) \neq 0   \}. 
\end{equation}
Then the chemical system  $(\Omega, \mathcal R , \mathcal R_s ,  \mathcal K, \mathcal K_s )$ satisfies the property of detailed balance. 
 \item    Assume that the chemical system $(\Omega, \mathcal R , \mathcal R_s ,  \mathcal K, \mathcal K_s )$ satisfies the detailed balance condition and 
 \begin{equation} \label{ineq cycles}
 |\mathcal W| >  |\mathcal W_r |. 
 \end{equation}
Then either \begin{enumerate}
\item the chemical rates in the set $ \mathcal K_s =\{ \alpha_k , \beta_k \}_{ k \in  \Omega_c} $ satisfy \eqref{eq for the parameters}  
\item or the property of detailed balance is unstable, i.e. for every $\delta >0$ there exists a set $ \mathcal K^\delta =\{ K_\rho^\delta , K_{- \rho}^\delta \}_{\rho=1}^r $ such that 
\[ 
\max_{ k \in \Omega } \left( |K_\rho - K_\rho^\delta | +  | K_{-\rho} - K_{-\rho}^\delta | \right)  \leq \delta
\]
and such that the chemical system  $(\Omega, \mathcal R , \mathcal R_s ,  \mathcal K^\delta, \mathcal K_s )$ does not satisfy the detailed balance property. 
\end{enumerate}
\end{enumerate}
\end{theorem}

The statement of Theorem \ref{cor:sources break detailed balance} can be summarized as follows. 
Assume that a chemical system satisfies the detailed balance property. Denote with $ \mathcal W_r$ the set of the cycles associated with this system. 
Assume that sources and sinks of substances in the set $ \Omega_c $ are added to the set of reactions of the system. If adding the sources and sinks does not produce new cycles, then the set $ \Omega_c^{\mathcal W} $ is empty. Therefore the chemical system with sources and sinks satisfies the detailed balance condition for any choice of sources and sinks rates. Otherwise if a cycle $w $ is created by adding sources and sinks (i.e. when $ |\mathcal W| > |\mathcal W_r| $), then the detailed balance property holds for the system with sources and sinks if  the sources and sinks rates of substances that belong to the cycle $w$ can be written in terms of the Gibbs free energy $E$ satisfying \eqref{energy reactions}, i.e. they must be given by \eqref{eq for the parameters}. 
Instead, if the sources and sinks rates are not of the form \eqref{eq for the parameters} then the detailed balance property can hold in the chemical system with sources and sinks only in an unstable manner. 
 In other words, if new cycles are created adding sources and sinks, then the sources and sinks rates need to be fine-tuned in order to obtain a  system with sources and sinks that satisfies the detailed balance condition.

We now discuss our results on the long time behaviour of the solution to the master equation. 
We start by stating the definition of weak solution to \eqref{eq:master with sources}. 

\begin{definition}[Weak solution] \label{def:solution master eq}
Let $f_0 \in \mathcal P (\mathbb N^M )$. 
A weak solution to equation \eqref{eq:master with sources} with initial datum $f_0$  is a function $f:  [0, \infty ) \rightarrow \mathcal P (\mathbb N^M) $ such that $f(\cdot, n ) $ is continuously differentiable for every $n \in \mathbb N^M $,  $f(0, \cdot)=f_0 $ 
and satisfies \eqref{eq:master with sources weak} for every test function $\varphi \in c_{00}(\mathbb N^M )$. 
\end{definition}

It is well known (see for instance \cite{anderson2018non}) that the continuous-time Markov chain modelling a chemical system can be explosive, i.e.  an infinite number of jumps can take place in finite time $T>0$.
This happens if and only if the solution $f$ to the master equation is such that 
\[ 
\sum_{n \in \mathbb N^M} f(t, n ) = 1 
\]
only if $0 < t \leq T  < \infty $ and for $t > T $ we have that 
\[ 
\sum_{n \in \mathbb N^M} f(t, n ) < 1. 
\]

It was proven in \cite{maas2020modeling} that chemical systems that satisfy the detailed balance property are not explosive and hence the solution to the master equation \eqref{eq:master with sources} is such that $f(t, \cdot ) \in \mathcal P (\mathbb N^M) $ for every time $ t \geq 0 $. The same result has been proven in \cite{anderson2018non} for systems that satisfy the complex balance property.  

One of the results that we prove in this paper is that stochastic chemical systems with sources and sinks are non-explosive if the chemical reactions are conservative.

\begin{remark}
Notice that conservative chemical systems with sources and sinks do not satisfy in general the complex balance property. 
\end{remark}

\begin{theorem}\label{thm:existence and uniqueness of time dep}
Assume that $(\Omega, \mathcal R, \mathcal R_s, \mathcal K , \mathcal K_s) $ is a chemical system with sources and sinks and that the set of reactions $\mathcal R $ is conservative. 
    There exists a weak solution $f:  [0, \infty ) \rightarrow \mathcal P (\mathbb N^M) $ to equation \eqref{eq:master with sources weak} in the sense of Definition \ref{def:solution master eq}. If the chemical system $(\Omega, \mathcal R, \mathcal R_s, \mathcal K , \mathcal K_s) $ is a mass action chemical system then the solution is unique. 
 \end{theorem}
Notice that the existence result in Theorem \ref{thm:existence and uniqueness of time dep} holds also for non mass action chemical systems. 
\begin{remark}
    As will be explained in more detail in Section \ref{sec:existence} (see Theorem \ref{thm:existence of time dep non rev}) the existence of a time dependent solution to the  master equation holds also for chemical systems with sources and sinks where the chemical reactions are not necessarily reversible, i.e. when we consider the master equation 
    \begin{align} \label{eq:master with sources non reversible}
    \partial_t f (t, n) &= \sum_{ \rho =1}^r R_\rho(n - S_\rho ) f(t,n-S_\rho )  -f(t, n)  \sum_{ \rho =1}^r R_\rho (n)  \\
    & + \sum_{ k \in \Omega_c } [ A_k(n - e_k )   f(t,  n - e_k ) + B_k(n+e_k) f(t, n+e_k) ] \nonumber \\
    &-  \sum_{ k \in \Omega_c } [ A_k(n)  + B_k(n)  ] f( t, n  ). \nonumber 
\end{align}
    and the chemical reactions $ \{ S_\rho \}_{ \rho =1}^r$ are conservative. 
\end{remark}

 We will then study the long-time behaviour of the solution the master equation. In particular we prove that it converges to a unique steady state as time tends to infinity. 
\begin{theorem} \label{thm: convergence to steady state}
    Assume that $(\Omega, \mathcal R, \mathcal R_s, \mathcal K , \mathcal K_s) $ is a mass action chemical system with sources and sinks and that the set of reactions $\mathcal R $ is conservative.
   Then there exists a unique steady state $ \overline  f \in \mathcal P(\mathbb N^M ) $ to equation \eqref{eq:master with sources weak}. 
    Moreover, we have that
    \begin{equation}
        f(t, \cdot  ) \rightharpoonup \overline f \text { as }  t \to \infty \text{ in the } \ast-\text{weak  topology}. 
    \end{equation}
\end{theorem} 

Notice that, as proven in Theorem \ref{cor:sources break detailed balance}, the system obtained adding sources and sinks to a chemical system that satisfies the detailed balance property does not, in general, satisfy the detailed balance property. As a consequence the stationary solution $ \overline f $ is a non-equilibrium stationary solution that cannot be written as an exponential of the Gibbs free energy. This is in contrast with what we will see for equilibrium steady states that can be written as exponentials of the Gibbs free energy, see \eqref{steady state when detailed balance holds}.

\section{Lack of detailed balance induced by the sources and sinks} \label{sec:lack of detailed balance when sources and sinks}
In this section we show that if we add sources/sinks to a chemical system that satisfies the detailed balance property, then we obtain a chemical system that does not satisfy the property of detailed balance, unless the rates of the sources and of the sinks are selected in a very specific way or if the chemical system that we obtain by adding the  sinks and sources do not have cycles. We refer to Theorem \ref{cor:sources break detailed balance} for the precise statement. 
In this section we will consider only mass action systems.

We now summarize some important known characterizations of chemical systems and the stationary solutions to the master equation of chemical systems that satisfy the detailed balance property. 
\begin{lemma}
The following statements hold for mass action systems. 
\begin{enumerate} 
\item A chemical system $(\Omega, \mathcal R , \mathcal R_s , \mathcal K, \mathcal K_s )$ satisfies the detailed balance property  if and only if 
\begin{equation} \label{Weg}
\prod_{ \rho =1  }^r  \left( \frac{K_\rho }{ K_{-\rho} } \right)^{w(\rho)} \prod_{ k \in \Omega_c   }^r  \left( \frac{\alpha_k }{\beta_k  } \right)^{w(k )}  =1, \quad  \forall w \in \mathcal W.
\end{equation}
\item 
If the detailed balance condition holds at the steady state $f_s \in \mathcal P ( \mathbb N^M )  $, then 
    \begin{equation} \label{steady state when detailed balance holds}
         f_s(n)= \frac{1}{Z} \prod_{k=1}^M   \frac{ e^{ - E_k n_k } }{n_k ! } =\frac{1}{Z}  \frac{ e^{ - E^T  n } }{n ! }, \quad n \in \mathbb N^M 
     \end{equation}
     where $E \in \mathbb R^M $ is a solution to the system of equations
         \begin{equation} \label{eq:energies}
            S_\rho^T  E = \ln \left( \dfrac{ K_{ - \rho} }{ K_{  \rho } }  \right)  \ \forall \rho \in \{ 1, \dots, r \}, \quad     e_k^T  E = \ln \left( \dfrac{\beta_k  }{ \alpha_k  }  \right)  , \quad \forall k  \in \Omega_c. 
        \end{equation}
     and $Z:= \prod_{\ell =1}^M e^{e^{- E_\ell} }$.
     \item 
     If detailed balance holds, then every steady state $\overline f $ to equation \eqref{eq:master with sources}, i.e. every solution to equation \eqref{eq:master with sources_stst}, has the form \eqref{steady state when detailed balance holds} for a solution $E \in \mathbb R^M $ to the system of equations \eqref{eq:energies}. 
     \end{enumerate}
    \end{lemma} 
    \begin{proof}
    We start by proving 1. 
        First of all notice that the condition \eqref{Weg} holds if and only if there exists a vector $E \in \mathbb R^{ M } $ that is satisfies \eqref{eq:energies}. 
   If  we define $f_s $ as \eqref{steady state when detailed balance holds} then we have that 
     \begin{align*}
     R_\rho (n) f_s (n) &= \dfrac{ K_\rho  }{Z (n- \nu_\rho) !  } e^{ - E^T n } = \dfrac{ K_\rho  }{Z (n- \nu_\rho) !  } e^{ - E^T n } = \dfrac{ K_{ - \rho} }{Z (n- \nu_\rho) !  } e^{ - E^T (n+ S_\rho )  } \\
     & = R_{ - \rho  } ( n + S_\rho ) f_s( n+ S_\rho), \quad  \forall \rho \in \{ 1, \dots, r \} . 
     \end{align*}
     Similarly 
        \begin{align*}
     A_k (n) f_s (n) &=  \dfrac{ \alpha_k  }{ Z n !  } e^{ - E^T n }  = \dfrac{ \beta_k  }{Z n  !  } e^{ - E^T (n+ e_k)  } = B_k (n) f_s(n + e_k ), \quad \forall k \in \Omega_c. 
     \end{align*}
     Hence \eqref{Weg} implies that detailed balance holds. 

     On the other hand if we assume that detailed balance holds we have that $ \forall \rho \in \{ 1, \dots, r \} $ and for all $ k \in \Omega_c $ 
     \begin{equation} \label{eq: g }
\dfrac{K_{ \rho }}{ K_{- \rho } }   = \dfrac{g (n + S_\rho ) }{ g (n) },  \quad \dfrac{\alpha_k }{ \beta_k  }   = \dfrac{g (n + e_k ) }{ g (n) } \quad  \forall n \in \mathbb  N^M 
     \end{equation}
     where $ g (n) = n ! f_s (n) $. 

     Assume that $w \in \mathcal W $. Then, using \eqref{eq: g } and the definition of cycles we deduce that
     \begin{align*}
         g (0) &= g \left(  \sum_{\rho=1}^{r } w(\rho) S_\rho + \sum_{ k \in \Omega_c } w(k ) \right) =   \left( \dfrac{K_r }{ K_{ - r } }  \right)^{ w (r)} g \left(  \sum_{\rho=1}^{r -1  } w(\rho) S_\rho + \sum_{ k \in \Omega_c } w(k ) \right) \\
         & = \prod_{ \rho =1 }^r  \left( \dfrac{K_\rho }{ K_{ - \rho} }  \right)^{ w (r)} g \left(  \sum_{ k \in \Omega_c } w(k ) \right) = \prod_{ \rho =1 }^r  \left( \dfrac{K_\rho }{ K_{ - \rho} }  \right)^{ w (r)} \prod_{ k \in \Omega_c }  \left( \dfrac{ \alpha_k  }{ \beta_k }  \right)^{ w (k )} g (0) 
     \end{align*}
     Therefore \eqref{Weg} holds. 
    The proof of point 2. and 3. can be found in \cite{anderson2010product}, Theorem 4.2. 
    \end{proof}
The condition \eqref{Weg} is usually called \textit{Wegscheider criterion} (see \cite{wegscheider1902simultane}) or Kolmogorov's criterion. 
\begin{remark}
    Notice that the solution to the system of equations \eqref{eq:energies} is not unique. Indeed assume that $E \in \mathbb N^M  $ is a solution, then also the vector $E_m := E+ m $  where $ m \in \mathcal M $ is a solution. 
\end{remark}

\begin{remark}
The normalization constant $Z = \exp \left( - \sum_{ k \in \Omega } E_k \right)  $  is needed in order to have that 
\[ 
\sum_{ n \in \mathbb N^M }  f_s (n)= 1. 
\] 
\end{remark}

Let denote with $\mathcal B_{\mathcal W} $ the basis of the space of the cycles. 
Motivated by the Wegscheider criterion \eqref{Weg} we define the parameters $\Delta(w)  $ associated to the cycles $w \in  \mathcal B_{\mathcal W} $ as 
\begin{equation} \label{parameter delta}
    \Delta (w):=\sum_{ \rho=1  }^r  w (\rho) \log \left( \frac{K_\rho }{ K_{-\rho} } \right)
\end{equation}
These parameters measure the lack of detailed balance along the cycles of the system. Detailed balance holds if and only if $ \Delta (w )=0$ for every $ w \in \mathcal B_{ \mathcal W } $. 

Our goal is to show that adding sources and sinks to a kinetic system $(\Omega, \mathcal R,  \mathcal K) = (\Omega, \mathcal R,  \emptyset , \mathcal K, \emptyset)$ that satisfies the detailed balance property, will break the detailed balance property unless the sources and sinks rates are fine tuned or the sources and sinks, as well as the chemical reactions, satisfy a topological condition. In order to formulate this topological condition we compare the cycles of the system with and without sources and sinks.   
\begin{lemma}
   Assume that $(\Omega, \mathcal R,  \mathcal R_s , \mathcal K, \mathcal K_s )$ is a chemical system with sources and sinks. Assume that the chemical reactions $ \mathcal R $ are conservative. 
   Assume that $\mathcal W $ is the space of the cycles associated with $(\Omega, \mathcal R,  \mathcal R_s , \mathcal K, \mathcal K_s )$. 
On the other hand, assume that $\mathcal W_r$ is the space of the cycles associated with $(\Omega, \mathcal R,  \mathcal K) = (\Omega, \mathcal R,  \emptyset , \mathcal K, \emptyset)$. 
 Let $w \in \mathcal W_r$.  Then we have that the vector  $w^r =(w^T, 0 )^T  \in \mathcal  W$. 
\end{lemma}
\begin{proof}
    It follows from the definition of cycles. 
\end{proof}

\begin{remark}
The lemma above implies that $|\mathcal W | \geq |\mathcal W_r |. $ 
We stress that we can have situations in which $ |\mathcal W_r | >  |\mathcal W | $. See for instance the system in Section \ref{sec:example 1} and the system in Section \ref{sec:example 2}. 
\end{remark}

\begin{example} \label{example: no additional cycles}
We now present a simple example of chemical systems to which we add sources and sinks in such a way that $|\mathcal W | =  |\mathcal W_r |. $ 
Consider the chemical system 
\begin{equation} \label{simple cycle}
(1 ) \leftrightarrows (2)  \leftrightarrows (3)  \leftrightarrows (1). 
\end{equation}
Clearly this chemical system has a cycle $ w=(1,1,1)^T $.
We now consider the chemical system with a source and sink of chemicals of type $(1) $, i.e.  
\begin{equation} \label{simple cycle with one source}
(1 ) \leftrightarrows (2)  \leftrightarrows (3)  \leftrightarrows (1)\leftrightarrows \emptyset. 
\end{equation}
All the cycles of this system are of the form $ (w, 0)^T  $ for some cycle $ w $ of the system without sources and sinks, i.e. \eqref{simple cycle}. 
\end{example}

In the following proposition we compute the parameter $ \Delta $, defined as in  \eqref{parameter delta}, measuring the amount of lack of detailed balance in a chemical system $(\Omega, \mathcal R , \mathcal R_s ,  \mathcal K, \mathcal K_s )$ that is obtained adding sources and sinks to a chemical system $(\Omega, \mathcal R , \emptyset ,  \mathcal K, \emptyset )$ that satisfies the detailed balance condition. 
\begin{proposition} \label{prop:lack of detailed balance}
    Assume that $(\Omega, \mathcal R , \mathcal R_s ,  \mathcal K, \mathcal K_s )$ is a chemical system with sources and sinks. Assume that the system $(\Omega , \mathcal R, \mathcal K ) =(\Omega, \mathcal R , \emptyset ,  \mathcal K, \emptyset ) $ without sources and sinks satisfies the detailed balance condition.  Let $\mathcal B_{\mathcal W_r} = \{ w_k \}_{k =1}^{|\mathcal W_r |} $ be a basis for the space of the cycles of $(\Omega , \mathcal R, \mathcal K )$ and let $\mathcal B_{ \mathcal W} =  \{ w_k \}_{k =1}^{|\mathcal W |} $ be a basis of the space of the cycles of  $(\Omega, \mathcal R , \mathcal R_s ,  \mathcal K, \mathcal K_s )$. Assume that 
 \begin{equation} \label{ineq cycles}
 |\mathcal W| >  |\mathcal W_r |. 
 \end{equation}
 Then for every $w \in \mathcal W$ such that $ \pi_{ \mathcal R } w \notin \mathcal W_r $, it holds that 
 \begin{equation} \label{eq:lack of DB Delta}
 \Delta (w)= (E- D)^T \mathbb S_r \pi_{\mathcal R }  w
 \end{equation}
where $E \in \mathbb R^M $ is a solution to \begin{equation} \label{energy reactions}
 S_\rho^T  E = \ln \left( \dfrac{ K_{ - \rho} }{ K_{  \rho } }  \right)  \ \forall \rho \in \{ 1, \dots, r \}   
\end{equation} 
and where $D \in \mathbb R^M $ is the vector 
\[
D(k)= \begin{cases} &\log \left( \frac{\alpha_k }{ \beta_k } \right), \quad k \in \{ 1, \dots, |\Omega_c | \}   \\
& 0, \quad   k \in \{|\Omega_c | +1 , \dots , M \}. 
\end{cases}
\]
\end{proposition}

\begin{proof}
By the definition \eqref{parameter delta} of $\Delta (w) $ we have that 
\begin{equation} \label{for DB}
\Delta(w)= \sum_{\rho=1}^r w(\rho) \log \left(  \frac{ K_\rho }{ K_{- \rho }} \right)  + \sum_{k=1}^{|\Omega_c|} w(k+ r +1 ) \log \left(  \frac{ \alpha_k }{ \beta_k } \right),  \ \forall w \in \mathcal W . 
\end{equation}
Let $w \in \mathcal W$ be such that $ \pi_{ \mathcal R } w \notin \mathcal W_r $ where we recall that the projection $\pi_{\mathcal R} w $ is defined as \eqref{proj reactions}. 
 Then $\mathbb S w =0$ where $\mathbb S $ is defined as \eqref{stoic matrix ext}.
Let us denote  with $\pi_{ \mathcal C } w \in \mathbb R^{|\Omega_c | } $ the vector defined as $ \pi_{\mathcal C} w (j )= w(1+r + j ) $ for $j  =1,  \dots , |\Omega_c | $. 
By the assumptions on $w $ we have that $ \pi_{ \mathcal R } w \notin \mathcal W_r  $, hence $\mathbb S_r \pi_{\mathcal R } w \neq 0$. Since $\mathbb S w =0$ this implies that $\mathbb S_r \pi_{\mathcal R } w = - \textbf I_{M \times |\Omega_c|} \pi_{\mathcal C } w  \neq 0$. Moreover since $(\Omega , \mathcal R , \mathcal K )$ satisfies the detailed balance property there exists a solution $ E \in \mathbb R^M $ to  \eqref{energy reactions}. 
Then the equality \eqref{for DB} can be rewritten as 
\begin{align*}
  \Delta (w) & =  \sum_{\rho=1}^r w(\rho) \log \left(  \frac{ K_\rho }{ K_{- \rho }} \right)  + \sum_{k=1}^{|\Omega_c|} w(j+ r +1 ) \log \left(  \frac{ \alpha_k }{ \beta_k } \right) = \sum_{\rho=1}^r w(\rho) S_\rho^T E  + \sum_{k=1}^{|\Omega_c|} \pi_{\mathcal C } w(k) \log \left(  \frac{ \alpha_k }{ \beta_k } \right) \\
  &= \sum_{k \in \Omega} \sum_{\rho=1}^r w(\rho) S_\rho(k) E(k)   + \sum_{k=1}^{|\Omega_c|} \pi_{\mathcal C } w(k) \log \left(  \frac{ \alpha_k }{ \beta_k } \right) = E^T \mathbb S_r \pi_{\mathcal R } w + D^T \textbf{I}_{M \times |\Omega_c |} \pi_{\mathcal C } w \\
  & =(E- D)^T \mathbb S_r \pi_{\mathcal R }  w.  
\end{align*}
\end{proof}
We now prove that if we add sources and sinks to a chemical system $(\Omega, \mathcal R , \emptyset ,  \mathcal K, \emptyset )$ that satisfies the detailed balance condition, then the resulting chemical system with sources and sinks does not satisfy the detailed balance property, unless the chemical rates of the sources and sinks are chosen in a specific way or the addition of sources and sinks does not produce cycles in the system $(\Omega, \mathcal R , \emptyset ,  \mathcal K, \emptyset )$.

\begin{remark}
    We stress that if adding sources and sinks does not create new cycles, the detailed balance property of the system with sources and sinks is inherited by the system with sources and sinks. This is the case in Example \ref{example: no additional cycles}. System \eqref{simple cycle with one source} satisfies the detailed balance property if system \eqref{simple cycle} satisfies the detailed balance property. 
\end{remark}
\begin{proof} [ Proof of Theorem \ref{cor:sources break detailed balance} ] 
By the Wegscheider criterion we know that the chemical system $(\Omega, \mathcal R , \mathcal R_s, \mathcal K, \mathcal K_s )$ satisfies detailed balance if and only if $\Delta (w)=0$ for every $ w \in \mathcal W $. Notice that the fact that the system $(\Omega, \mathcal R, \mathcal K ) $ satisfies the detailed balance property implies that $\Delta (w)=0$ for every $ w \in \mathcal W$.  Hence the detailed balance property holds for the extended system if and only if 
\[ 
\Delta (w)=0, \  \forall w \in \mathcal W.  
\]
By Proposition \ref{prop:lack of detailed balance} we deduce that if \eqref{eq for the parameters} holds  then 
\begin{align*}
\Delta (w) = (E- D)^T \mathbb S_r \pi_{\mathcal R }  w =0 \  \forall w \in \mathcal  W. 
\end{align*}
Hence point 1. in Theorem \ref{cor:sources break detailed balance}  holds. 

We now prove the second point. 
Assume now that \eqref{eq for the parameters} does not hold and that $(\Omega, \mathcal R , \mathcal R_s, \mathcal K, \mathcal K_s )$ satisfies the detailed balance condition. 
Since  $|\mathcal W | > |\mathcal W_r |$ we know that there exists a $ \tilde w \in \mathcal W $ such that  $\mathbb S_r \pi_{\mathcal R} \tilde  w \neq 0 $. Therefore we can define $ \ell \in \Omega $ be such that  $\mathbb S_r \pi_{\mathcal R} \tilde w (\ell) \neq 0  $. Let $\delta >0 $  and let $E $ be a solution to \eqref{energy reactions}. 
We construct the perturbed rates as follows
\begin{align*}
K_{\rho}^\delta &= K_\rho, \quad K_{\rho}^\delta = K_{\rho} \exp \left(  \sum_{ k \in \Omega } E_\delta (k)  S_\rho (k) \right)
\end{align*}
where $E_\delta (k )=0 $ for every $ k \neq \ell $ and $E_\delta (\ell)\neq 0 $. Notice that $\tilde E= E + E_\delta $ satisfies \eqref{energy reactions} by construction.
The detailed balance property of $( \Omega, \mathcal R, \mathcal R_s, \mathcal K , \mathcal K_s) $ implies that $(  E - D )^T \mathbb S_r \pi_{\mathcal R} w =0$ for every $ w \in \mathcal W $. Therefore  
\begin{align*}
    ( \tilde E - D )^T \mathbb S_r \pi_{\mathcal R} \tilde w = E_\delta ^T \mathbb S_r \pi_{\mathcal R} \tilde w = E_\delta (\ell) (\mathbb S_r \pi_{\mathcal R} \tilde w) (\ell)  \neq 0. 
\end{align*}
We conclude that the perturbed system $( \Omega, \mathcal R, \mathcal R_s, \mathcal K , \mathcal K_s) $  does not satisfy the detailed balance property. Notice that we can select $E_\delta ( \ell) $ arbitrarily small, hence the point 2. of Theorem \ref{cor:sources break detailed balance} holds.

\end{proof}

\section{Existence of a unique time dependent solution to the master equation} \label{sec:existence}
In this section we prove that the master equation \eqref{eq:master with sources} has a unique weak time dependent solution if the chemical reactions are conservative.

 \subsection{Existence} 
 The goal for this section is to prove the existence of a time dependent solution to the master equation \eqref{eq:master with sources}, i.e. we aim at proving the following theorem. 
 \begin{theorem}\label{thm:existence of time dep}
Assume that $(\Omega, \mathcal R, \mathcal R_s, \mathcal K , \mathcal K_s) $ is a chemical system with sources and sinks and that the set of reactions $\mathcal R $ is conservative. Assume moreover that $f_0 \in \mathcal P (\mathbb N^M)$ is such that 
\[ 
\sum_{ n \in \mathbb N^M } |n| f_0(n) < \infty . 
\]
    There exists a weak solution $f:  [0, \infty ) \rightarrow \mathcal P (\mathbb N^M) $ to equation \eqref{eq:master with sources weak} in the sense of Definition \ref{def:solution master eq}.
 \end{theorem} 
In order to prove Theorem \ref{thm:existence of time dep} it is convenient to study a truncated problem in which the kinetics are supported in a finite set of states. 
We then prove the existence of a unique solution to the truncated problem and find suitable uniform estimates on the solution that allow to remove the truncation and to prove the existence of a solution to \eqref{eq:master with sources weak}. 
To this end we define the truncated kinetics as follows.

Let $N \in \mathbb N$. We define $I:=\{ 1, 2, \dots, N \} $. Then for every chemical reaction $S_\rho $ we define the truncated rate $R_\rho^N: \mathbb Z^M \rightarrow \mathbb R_* $ as
\begin{equation} \label{truncated rates}
R_\rho^N (n) :=\begin{cases} 
R_\rho (n) & \text{ if }  n \in I^M  \\
  0 &\text{ otherwise }
\end{cases}.
\end{equation}
Moreover we truncate the source rate as follows
\begin{equation} \label{truncated source}
    A_k^N  (n)  = \begin{cases} & \alpha_k, \text{ if } n \in I^M, \\
    & 0 \ \text{ otherwise }.
    \end{cases}
 \end{equation}

The first goal is to prove the existence of strong solutions $f_N$ to the truncated master equation 
\begin{align} \label{eq:master without sources truncated}
    \partial_t f_N (t, n) &= \sum_{ \rho=1   }^r  R^N_{-\rho}(n + S_\rho ) f_N (t,n+S_\rho ) +  \sum_{ \rho =1  }^r R_{\rho}^N ( n - S_\rho ) f_N(t, n-S_\rho )  \\
    & -f_N(t, n)  \sum_{ \rho =1}^r(R_\rho^N (n) + R^N_{- \rho} (n) ) +  \sum_{ k \in \Omega_c } [ A^N_k(n - e_k ) f_N( n - e_k ) + B_k(n+e_k ) f_N(t, n+e_k) ] \nonumber \\
    &- f_N( t, n  ) \sum_{ k \in \Omega_c } [ A^N_k(n) + B_k(n)  ]  \, \  t >0, \, \quad f(0, n   ) = f_0(n), \  n \in \mathbb N^M \nonumber. 
\end{align}
\begin{proposition} \label{prop:existence truncated} 
  Let $N >0 $ and let $R^N_\rho $ be given by \eqref{truncated rates} and $A^N_k $ to be given by  \eqref{truncated source}. Assume $B_k $ to be given by \eqref{mass action sink}. 
  Let the initial datum $f_0 \in \ell^1 (I^M) $ be such that $f_0(n) \geq 0$ for every $n \in I^M $ and 
$
  \sum_{n \in I^M } f_0(n) \leq 1. 
$
Then, there exists a unique solution $ f_N (t, n )  $ with $ f_N( \cdot , n  ) : [0, \infty) \mapsto \mathbb R_+ $ continuously differentiable for any $n \in \mathbb N^M $, which satisfies \eqref{eq:master without sources truncated} in the classical sense. Moreover it is such that
\begin{equation} \label{bound numbers truncated}
  \sum_{n \in \mathbb N^M } f_N(t, n ) =  \sum_{n \in I^M } f_0(n) \leq 1, \quad \forall t \geq 0
\end{equation} 
and 
\begin{equation} \label{compact support}
   f_N(t, n ) = 0 , \quad \forall n \in \overline I^N \quad  \forall t \geq 0
  \end{equation}
  where $\overline I := [0, 1+ N +  \max_{\rho=1}^r \max_{k \in \Omega } | S_\rho (k) |  ]$. 
\end{proposition}
\begin{proof}[Proof of Proposition \ref{prop:existence truncated}]
    We start  proving that if $f_N $ is a solution to \eqref{eq:master without sources truncated}, then \eqref{compact support} holds. 
    In order to prove this, notice that
    \begin{align*}
    \sum_{n \notin \overline I^N } f_N (t, n) &=     \sum_{n \notin \overline I^N } \left[ \sum_{ \rho=1   }^r  R^N_{-\rho}(n + S_\rho ) \int_0^t  f_N (s,n+S_\rho ) \right] + \sum_{n \notin \overline I^N } \left[ \sum_{ \rho =1  }^r R_{\rho}^N ( n - S_\rho )\int_0^t  f_N(s, n-S_\rho ) \right]  \\
    & - \sum_{n \notin \overline I^N } \left[ \sum_{ \rho =1}^r(R_\rho^N (n) + R^N_{- \rho} (n) ) \int_0^t f_N(s, n) ds   \right] \\
    & + \sum_{n \notin \overline I^N } \sum_{ k \in \Omega_c } [ A^N_k(n) \int_0^t f_N( s, n - e_k ) ds + B_k (n+e_k)  \int_0^t f_N(s, n+e_k)  ds ] \\
    & - \sum_{n \notin \overline I^N } \sum_{ k \in \Omega_c } [ A^N_k(n) + B_k ( n )  ] \int_0^t f_N( s, n  ) ds \\
    & =      \sum_{n \notin \overline I^N } \sum_{ k \in \Omega_c } B_k(n+ e_k) \int_0^t f_N(s, n+e_k)  ds - \sum_{n \notin \overline I^N } \sum_{ k \in \Omega_c }  B_k(n)   \int_0^t f_N( s, n  ) ds \leq 0 .
    \end{align*}
Hence \eqref{compact support} holds since $f_0 \in \ell^1 (I^M) $. 
In order to prove the existence of a time dependent solution to equation \eqref{eq:master without sources truncated} we rewrite the equation \eqref{eq:master without sources truncated} as 
\[ 
\frac{d f_N(t, n ) }{dt} = T_n ( f_N ) , \quad  n \in \overline I^M 
\]
where $T_n :  \mathbb R_+^{ \overline I^M} \rightarrow \mathbb R_*  $ is defined as 
\begin{align*} 
T_n ( x ) :=&  \sum_{ \rho =1 }^r R^N_{-\rho}(n + S_\rho ) x(n+S_\rho ) +  \sum_{ \rho=1 }^r  R_{\rho}^N ( n - S_\rho ) x( n-S_\rho )  -x(n)  \sum_{ \rho =1}^r(R_\rho^N (n) + R^N_{- \rho} (n) ) \\ 
&   +  \sum_{ k \in \Omega_c } [ A^N_k(n - e_k ) x( n - e_k ) + B_k(n+e_k ) x( n+e_k) ] -  \sum_{ k \in \Omega_c } [ A^N_k(n) + B_k(n)  ] x (  n  )
\end{align*}
Notice that since the chemical rates $R_\rho^N$, $A^N_k$ and $B_k $ are bounded in $\overline I^M $ the functions $T_n $ are locally Lipschitz continuous. Using the Picard–Lindelöf theorem we deduce that there exists a unique  solution to \eqref{eq:master without sources truncated} on a time interval $ [0, T_*] $.

We now prove that \eqref{bound numbers truncated} holds. This is a consequence of  the definition of the truncated kernels $A^N_k $ and $R^N_\rho $ and the definition of $B_k $. Indeed we have that 
\begin{align*}
    \sum_{ n \in \overline I^M} \sum_{ \rho=1   }^r  R^N_{-\rho}(n + S_\rho ) f_N (t,n+S_\rho )  & =     \sum_{ \rho=1   }^r \sum_{\{ n \in \overline I^M :  n + S_\rho \in I^M \} }  R^N_{-\rho}(n + S_\rho ) f_N (t,n+S_\rho ) \\
     &= 
      \sum_{ \rho=1   }^r \sum_{  n  \in I^M  }  R^N_{-\rho}(n  ) f_N (t,n ). 
\end{align*}
As well as 
\begin{align*}
    \sum_{ n \in \overline I^M} \sum_{ \rho=1   }^r  R^N_{\rho}(n - S_\rho ) f_N (t,n-S_\rho )  = 
      \sum_{ \rho=1   }^r \sum_{  n  \in I^M  }  R^N_{\rho}(n  ) f_N (t,n )
\end{align*}
and 
\[
 \sum_{ n \in \overline I^M}  \sum_{ k \in \Omega_c } A^N_k(n - e_k ) f_N( n - e_k ) =  \sum_{ n \in \overline I^M}  \sum_{ k \in \Omega_c }  A^N_k(n)  f_N( t, n  ). 
\]
Finally we also know that by the definition of $B_k $ we have that 
\begin{align*}
    \sum_{ k \in \Omega_c }    \sum_{ n + e_k \in \overline I^M}  B_k(n+e_k ) f_N(t, n+e_k) &=   \sum_{  n \in \overline I^M }  \sum_{ k \in \Omega_c }   \beta_k (n_k+1) f_N(t, n+e_k)     \\
     &= \sum_{\{  n \in \mathbb N^M : n - e_k \in \overline I^M  \}  }  \sum_{ k \in \Omega_c }   \beta_k n_k f_N(t, n)  =    \sum_{ n  \in \overline I^M}  \sum_{ k \in \Omega_c } B_k(n)   f_N( t, n  ). 
\end{align*}
Summarizing we have that 
\begin{align*}
       \dfrac{d }{dt}\sum_{ n \in \overline I^M} f_N (t, n) &=  0  \nonumber. 
\end{align*}
Therefore \eqref{bound numbers truncated} holds. 
The bound \eqref{bound numbers truncated} also implies that the solution $f_N$ on the time interval $[0, T_*] $ can be extended to a global solution in time.   
\end{proof}
We now use Proposition \ref{prop:existence truncated} in order to prove the existence of a time dependent solution to \eqref{eq:master with sources weak}.

\begin{proof}[Proof of Theorem \ref{thm:existence of time dep}]
We divide the proof of the statement in three steps. As a first step we prove that the sequence of solutions to the truncated problems $ \{f_N \}_{ N \geq 1}  $ converge to a limit $f $ as $N \to \infty $. To this end we will use the uniform bound $ \| f_N (t, \cdot) \|_1 \leq 1   $ for every $N$ and for every $t \geq 0$. As a second step we will show that the limit $f$ satisfies equation \eqref{eq:master with sources weak}. As a third step we will use the fact that the chemical reactions are conservative to deduce that for every $t \geq 0 $
 we have that
 \begin{equation} \label{eq:bound on the mass}
 \sum_{  n \in \mathbb  N^M  } |n|  f(t, n) \leq C_0. 
 \end{equation}
This is the crucial step that allows us to prove that $ \| f(t, \cdot) \|_1 =1 $ for every $t \geq 0$.

\textit{Step 1. Existence of the limit.}
  Consider the initial condition $f_N(0,n ) = ( f_0 (n) )_{n \in I^M} $ for equation \eqref{eq:master without sources truncated}. 
Proposition \ref{prop:existence truncated} guarantees the existence of a sequence of solutions $ \{ f_N\}_{N \geq 1 } $ to equation \eqref{eq:master without sources truncated}.
Notice that for every $N \geq 1$ and every $T \geq 0 $ and every $n \in \mathbb N^M $ we have that the function $  t \mapsto f_N (t, n ) $ is continuous and satisfies the equation in weak form, i.e. it satisfies
    \begin{align} \label{eq:weak form master trunc}
    \sum_{  n \in \mathbb  N^M  } \varphi(n)  f_N (t, n) &= \sum_{  n \in \mathbb  N^M  } \varphi(n)  f_N (0,  n) \\
    &+ \sum_{ n  \in \mathbb  N^M } \mathcal L^N_r[ \varphi](n)  \int_0^t  f_N(s,n) ds  +  \sum_{ n  \in \mathbb  N^M } \mathcal L^N_s [ \varphi](n) \int_0^t  f_N(s,n) ds \nonumber
\end{align}
for every test function $\varphi \in c_{00}(\mathbb N^M )$. 
Here we use the notation
\begin{align} \label{truncated dual reaction}
\mathcal L^N_r   [\varphi] (n )  : =  \sum_{ \rho =1}^r  R^N_\rho(n) \left[ \varphi(n+ S_\rho) - \varphi(n)  \right] +   \sum_{ \rho =1}^r   R^N_{-\rho}(n) \left[ \varphi(n- S_\rho) - \varphi(n)  \right]  \nonumber \\
\end{align}
and
\begin{align} \label{truncated dual sources}
     \mathcal L^N_s  [\varphi] (n )  :=       \sum_{ k \in \Omega_c }  A^N_k (n)   \left[ \varphi(n+ e_k ) - \varphi(n)  \right] +    \sum_{ k \in \Omega_c } B_k (n)   \left[ \varphi(n-  e_k ) - \varphi(n)  \right]. 
\end{align}
We now want to prove that there exists a function $f:  [0, \infty ) \rightarrow \mathcal P (\mathbb N^M) $ such that $f(\cdot, n ) $ is continuous for every $n \in \mathbb N^M $ and such that for every $n \in \mathbb N^M $ it holds that $f_N(t, n )  \rightarrow f (t,n ) $  as $N \to \infty $ uniformly for $t \in [0, T] $. 

To this end we notice that the bound \eqref{bound numbers truncated} implies that the sequence of functions $\{ f_N (t) \}_{N \geq 1 } $ is uniformly bounded for every $t \in [0, T] $.
We now prove that the sequence $\{ f_N (t) \}_{N \geq 1 } $  is equicontinuous in the $\ast$-weak topology of $ \ell^1 (\mathbb N^M)$. This follows by the fact that for every $ \varphi \in c_{0 0} $ it holds that
    \begin{align*} 
    \sum_{  n \in \mathbb  N^M  } \varphi(n)  | f_N (t_1, n) - f_N(t_2, n) | & \leq  
     \sup_{ s \in [0, T] } \sup_{n \in \mathbb N^M } f_N(s,n) |t_1 - t_2 | \cdot \\
    & \cdot   \sum_{ n  \in \mathbb  N^M } \left[ \sum_{ \rho =1}^r  \left( R_\rho^N(n) \left[  \varphi(n+ S_\rho) - \varphi(n)  \right]  +   R^N_{-\rho}(n) \left[ \varphi(n- S_\rho) - \varphi(n)  \right] \right)  \right. \\
  &  \left. +  \sum_{ k \in \Omega_c }  A^N_k (n)   \left[ \varphi(n+ e_k ) - \varphi(n)  \right] +    \sum_{ k \in \Omega_c } B_k (n)   \left[ \varphi(n-  e_k ) - \varphi(n)  \right] \right]  \nonumber \\
    & \leq  |t_1 - t_2 | C_\varphi
\end{align*}
where the constant $C_\varphi$ does not depend on $N$. 
Using Ascoli-Arzelà theorem we deduce that, up to selecting a subsequence of $\{ f_N \}_{ N \geq 1 } $, there exists a $f \in C([0, T] , \ell^1_{\text{weak}} (\mathbb N^M) ) $ such that $f_N(t)   {\rightharpoonup} f (t) $ in the $\ast-$weak topology of $\ell^1 (\mathbb N^M )$ as $N \to \infty $ for every $t \in [0, T] $. 

By an iterative application of the Ascoli-Arzelà Theorem for a sequence of times  $T_k = k \in \mathbb N $ and taking the diagonal sequence we deduce that there exists a $f \in C([0, \infty) , \ell^1_{\text{weak}} (\mathbb N^M) ) $ such that $f_N(t)  \overset{\ast} {\rightharpoonup} f (t) $ as $N \to \infty $ uniformly for every  $t $ in compact sets.
Moreover, using the fact that \eqref{bound numbers truncated} holds we deduce that 
\[
  \sum_{n \in \mathbb N^M } f(t, n ) \leq 1, \quad \forall t \geq 0. 
\]

\textit{Step 2. The limit $f(t, n ) $ satisfies equation \eqref{eq:master with sources weak}.}
We now take the limit as $N \to \infty $ of all the terms in equation \eqref{eq:weak form master trunc}.
Clearly we have that for every $\varphi \in c_{00} (\mathbb N^M )$ it holds that 
\[
    \sum_{  n \in \mathbb  N^M  } \varphi(n)  f_N (t, n) \rightarrow    \sum_{  n \in \mathbb  N^M  } \varphi(n)  f (t, n), \quad t > 0 \text{ and }     \sum_{  n \in \mathbb  N^M  } \varphi(n)  f_N (0, n) \rightarrow    \sum_{  n \in \mathbb  N^M  } \varphi(n)  f_0 ( n) , \text{ as } N \to \infty. 
\]

In order to prove that $f$ satisfies \eqref{eq:master with sources weak} we want to prove that for every $ t >0 $ it holds that 
\[ 
\left|  \sum_{  n \in \mathbb  N^M  } \mathcal L_r^N  [\varphi] (n)  \int_0^t f_N (s, n) ds -  \sum_{  n \in \mathbb  N^M  }  \mathcal L_r  [\varphi] (n)  \int_0^t  f(s, n) ds   \right| \rightarrow 0 \text{ as } N \to \infty 
\] 
as well as 
\begin{equation} \label{eq:convergence of sources and sinks}
\left|  \sum_{  n \in \mathbb  N^M  } \mathcal L_s^N  [\varphi] (n)  \int_0^t  f_N (s, n) ds -  \sum_{  n \in \mathbb  N^M  }  \mathcal L_s  [\varphi] (n)  \int_0^t f(s, n)  ds \right| \rightarrow 0 \text{ as } N \to \infty 
\end{equation}
To this end notice that 
\begin{align*}
\left|  \sum_{  n \in \mathbb  N^M  }\mathcal L_r^N [\varphi] (n)   \int_0^t  f_N (s, n) ds-  \sum_{  n \in \mathbb  N^M  } \mathcal L_r [\varphi] (n)   \int_0^t  f(s, n) ds \right| & \leq  \sum_{  n \in \mathbb  N^M  }  | \mathcal L_r [\varphi] (n)  - \mathcal L_r^N [\varphi] (n) | \int_0^t  f_N (s, n )  ds \\
& +  \sum_{  n \in \mathbb  N^M  }  \mathcal L_r [\varphi] (n) \int_0^t  | f_N (s, n ) -  f (s, n ) | ds 
\end{align*}
Since the function $\varphi$ has compact support, then also the function $ \mathcal L_r [\varphi] $ has compact support. Therefore we have that 
\[
\sup_{ s \in [0, t ] } \sum_{  n \in \mathbb  N^M  }   \mathcal L_r [\varphi] (n)  | f_N (s, n ) -  f (s, n ) |  \rightarrow 0 \text{ as } N \to \infty. 
\]
Moreover we have that as $ N \to \infty $
\begin{align*}
  \sum_{  n \in \mathbb  N^M  }|  \mathcal L_r^N [\varphi] (n) - \mathcal L_r [\varphi] (n)  |  \int_0^t f_N (s, n ) ds  \leq &  t   \sum_{  n \in \mathbb  N^M  } \sum_{ \rho =1}^r  \left| (  R_\rho^N(n) - R_\rho( n) )  \left[ \varphi(n+ S_\rho) - \varphi(n)  \right] \right| \\
  & +  t  \sum_{  n \in \mathbb  N^M  }   \sum_{ \rho =1}^r \left| ( R^N_{-\rho}(n) - R_{-\rho}(n))  \left[ \varphi(n- S_\rho) - \varphi(n)  \right]  \right| \rightarrow 0. 
\end{align*}
The fact that \eqref{eq:convergence of sources and sinks} holds follows by similar arguments. 
We deduce that the function $f \in C((0, \infty ) , \ell^1_{\text{weak}} (\mathbb N^M ) ) $ satisfies the following equation 
    \begin{align*} 
    \sum_{  n \in \mathbb  N^M  } \varphi(n)  f (t, n) &= \sum_{  n \in \mathbb  N^M  } \varphi(n)  f (0,  n) + \sum_{ n  \in \mathbb  N^M } \mathcal L_r[ \varphi](n)  \int_0^t  f(s,n) ds  +  \sum_{ n  \in \mathbb  N^M } \mathcal L_s [ \varphi](n) \int_0^t  f( s,n) ds
\end{align*}
for every $\varphi \in c_{00}(\mathbb N^M ) $. 
In particular this implies that $ f \in  C^1 ((0, \infty ) , \ell^1_{\text{weak}} (\mathbb N^M ) ) $  and that $f$ satisfies  \eqref{eq:master with sources weak}. 

\textit{Step 3. For every $t \geq 0$ it holds that $ \sum_{n \in \mathbb N^M }  f(t, n )=1 $.}
To this end we first prove that for every $t \geq 0$ it holds that the bound \eqref{eq:bound on the mass} holds for a constant $C_0>0 $ that does not depend on $N$. 
In order to prove this we use the fact that the set of chemical reactions $\mathcal R $ is conservative. Hence there exists a $m \in \mathcal M $ that is such that $m \in \mathbb R_+^L $. 
Since $f_N (t, \cdot  ) $ is compactly supported we can consider the test function $\varphi_m (n) := m^T n $ in equation \eqref{eq:weak form master trunc}. We deduce that 
    \begin{align*} 
    \sum_{  n \in \mathbb  N^M  } m^T n  f_N (t, n) &= \sum_{  n \in \mathbb  N^M  } m^T n  f_N (0,  n) +  \sum_{ n  \in \mathbb  N^M } \mathcal L^N_s [ m^T n ] \int_0^t  f_N(s,n) ds \nonumber
\end{align*}
where we used the fact that 
\begin{align*}
\mathcal L^N_r (m^T n )  =  \sum_{ \rho =1}^r  R^N_\rho(n) \left[ m^T (n+ S_\rho) -  m^T n  \right] +   \sum_{ \rho =1}^r   R^N_{-\rho}(n) \left[ m^T (n- S_\rho) - m^T (n)  \right] =0. 
\end{align*}
We deduce that 
\begin{align} \label{invariant region}
      \sum_{  n \in \mathbb  N^M  } m^T n  f_N (t, n)  & =    \sum_{  n \in \mathbb  N^M  } m^T n  f_N (0,  n)  +      \sum_{  n \in \mathbb  N^M  } \int_0^t f_N(s, n ) ds \sum_{ k \in \Omega_c } \left(  A^N_k (n) - B_k (n) \right)    m^T   e_k \nonumber \\
      & \leq  \sum_{  n \in \mathbb  N^M  } m^T n  f_N (0,  n)  +      \sum_{  n \in \mathbb  N^M  } \int_0^t f_N(s, n ) ds \left( |\Omega_c| \max_{k \in \Omega_c }  \alpha_k  -  \sum_{ k \in \Omega_c } B_k (n) \right)    m^T   e_k \nonumber
      \\
      & \leq  \sum_{  n \in \mathbb  N^M  } m^T n  f_N (0,  n)  +     t  \gamma  -  \sum_{  n \in \mathbb  N^M  } \int_0^t f_N(s, n ) ds \sum_{ k \in \Omega_c } \beta_k n_k    m^T   e_k \nonumber \\
      & \leq  \sum_{  n \in \mathbb  N^M  } m^T n  f_N (0,  n)  +     t  \gamma  -  \delta \int_0^t  \sum_{  n \in \mathbb  N^M  } m^T  n  f_N(s, n ) ds 
\end{align}
where $\gamma :=  |\Omega_c| \max_{k \in \Omega_c }  \alpha_k ( m^T e_k ) >0 $ and $ \delta := \min_{\Omega_c } \beta_k >0  $. 
Using Gr\"onwall's Lemma we deduce that 
\[ 
     \sum_{  n \in \mathbb  N^M  } |n|  f_N (t, n) \leq \dfrac{1 }{ \min_{k \in \Omega_c } m^T e_k  } \sum_{  n \in \mathbb  N^M  } m^T n  f_N (t, n) \leq C_0 
\]
where
\[
C_0 := \dfrac{1 }{ \min_{k \in \Omega_c }  m^T e_k }  \left(  \sum_{  n \in \mathbb  N^M  } m^T n  f(0,  n)  + \frac{\gamma }{\delta } \right) 
\]
is a constant that does not depend on the truncation parameter $N$. 
We deduce that the inequality holds also when we pass to the limit as $N \to \infty $, i.e. 
\[ 
     \sum_{  n \in \mathbb  N^M  } |n|  f(t, n) \leq C_0. 
\]
We can now prove that 
\begin{equation} \label{eq:no gelation}
\sum_{  n \in \mathbb  N^M  }   f(t, n)  =  \sum_{  n \in \mathbb  N^M  }   f_0 (n) = 1    
\end{equation}
holds for every time $t \geq 0 $. 
To this end notice that if $ N \geq L -1 >0 $, then 
\begin{align*} 
\left| \sum_{  n \in \mathbb  N^M  }   f(t, n) - \sum_{  n \in \mathbb  N^M  }   f_0 (n)  \right|   \leq & \sum_{|n | \leq L } \left| f(t, n) - f_N (t, n )  \right| +\left| \sum_{n \leq L}  f_N (t, n )   - \sum_{n \in \mathbb N^M } f_0 (n) \right| \\
& + \sum_{ | n | >  L }  [f (t, n) +  f_N (t, n) ] \\
& \leq \sum_{|n | \leq L } \left| f(t, n) - f_N (t, n )  \right| + \sum_{n \in \mathbb N^M \setminus I^M  } f_0 (n) + \sum_{ | n | >  L }  f (t, n)  +  \sum_{ | n | >  L }  f_N (t, n) 
 \end{align*} 
Using the fact that the  numbers $N$ and $L$ are arbitrary and 
 \[
 \sum_{ \{  n \in \mathbb  N^M : |n | > L    \}  }   f(t, n) \leq \dfrac{1}{L}  \sum_{ n \in \mathbb  N^M  } |n|  f(t, n) \leq \dfrac{C_0 }{L} \quad  \text{ and } \   \sum_{ \{  n \in \mathbb  N^M : |n | > L    \}  }   f_N (t, n)  \leq \dfrac{C_0 }{L}, 
 \]
and  $ \sum_{ n \in \mathbb N^M } f_0 (n) \leq 1 $, as well as the fact  that $f_N(t)  \overset{\ast} {\rightharpoonup} f (t) $ as $N \to \infty $ we deduce that \eqref{eq:no gelation} holds.  
\end{proof}
We conclude this section stressing that in the proof of Theorem \ref{thm:existence of time dep} we did not use that the chemical reactions in the network are reversible.  Indeed we can prove an existence result also for chemical systems that contain non-bidirectional chemical reactions. 
\begin{definition}[Non-bidirectional chemical system] 
Assume that $\mathcal R$ is a set of chemical reactions i.e. 
 $\mathcal R := \{ S_\rho \}_{\rho=1}^r $. Assume that $ \mathcal R_s $ is a set of sources and sinks, i.e.  $\mathcal R_s := \{e_k \}_{k \in \Omega_c}  \cup  \{ - e_{ k  }\}_{\Omega_c }  $ with $ \Omega_c \subset \Omega$.  
 Assume that $\mathcal K := \{ R_\rho  \}_{ \rho =1 }^r $ is the set of the kinetics associated with the chemical reactions $S_\rho \in \mathcal R $ and $\mathcal K_s := \{ \alpha_k  \}_{ k \in \Omega_c }  \cup \{ \beta_k  \}_{ k \in \Omega_c } $ is the set of the chemical rates associated with the sources and sinks. 
 We say that $(\Omega, \mathcal R,  \mathcal R_s , \mathcal K, \mathcal K_s )$ is a non-bidirectional chemical system with sources and sinks. 
\end{definition}
 \begin{theorem}\label{thm:existence of time dep non rev}
Assume that $(\Omega, \mathcal R, \mathcal R_s, \mathcal K , \mathcal K_s) $ is a non-bidirectional chemical system with sources and sinks and that the set of reactions $\mathcal R $ is conservative. Assume moreover that $f_0 \in \mathcal P (\mathbb N^M)$ is such that 
\[ 
\sum_{ n \in \mathbb N^M } |n| f_0(n) < \infty . 
\]
    There exists a weak solution $f:  [0, \infty ) \rightarrow \mathcal P (\mathbb N^M) $ that satisfies 
    \begin{align*}
    \sum_{  n \in \mathbb  N^M  } \varphi(n) \partial_t   f (t, n) &=    \sum_{ \rho =1}^r  R_\rho(n) \left[ \varphi(n+ S_\rho) - \varphi(n)  \right] +    \sum_{ k \in \Omega_c }  A_k (n)   \left[ \varphi(n+ e_k ) - \varphi(n)  \right] \\
    & +    \sum_{ k \in \Omega_c } B_k (n)   \left[ \varphi(n-  e_k ) - \varphi(n)  \right] 
    \end{align*}
    for every $ \varphi \in c_{ 00 } ( \mathbb N^M) $. 
 \end{theorem}

 \subsection{Uniqueness}
 In this section we will restrict the attention to mass action chemical systems. 
In order to prove Theorem \ref{thm:existence and uniqueness of time dep} it remains to prove that equation \eqref{eq:master with sources weak} has a unique solution. 
To this end it is convenient to study the dual to  equation \eqref{eq:master with sources weak}, i.e. 
    \begin{align*} 
- \partial_t  \varphi(t, n)   = 
   \mathcal L_r[ \varphi (t) ](n) + \mathcal L_s [ \varphi (t) ](n), \quad t \in [0, T],\  n \in \mathbb N^M  
\end{align*}
with $\varphi (T) = \varphi_0 \in c_{00} (\mathbb N^M )$. 
In order to study this equation we make the change of variables $\psi(s, n ) :=  \varphi (T- t, n )$. Hence 
    \begin{align} \label{eq: dual}
 \partial_s  \psi(s, n)   = 
   \mathcal L_r[ \psi (s) ](n) + \mathcal L_s [ \psi (s) ](n), \quad s\in [0, T] \quad n \in \mathbb N^M \text{ and } \psi(0,  n) = \varphi_0    . 
\end{align}

We start by proving the existence of a solution $ \psi $ to the dual problem and in particular we will prove that the solution $ \psi $ satisfies an exponential bound. More precisely we prove the following statement. 
\begin{proposition}\label{prop:existence of solution to the dual problem}
Assume that $(\Omega, \mathcal R, \mathcal R_s, \mathcal K , \mathcal K_s) $ is a mass action chemical system with sources and sinks and that the set of reactions $\mathcal R $ is conservative.
    For any  $T>0 $ there exists a sequence of functions $\psi (t)=(\psi( t , n ))_{ n \in \mathbb N^M} $ that is such that for every $ n \in \mathbb N^M $ the function  $ t \mapsto \psi(t, n )  $ is continuously differentiable and satisfies  equation \eqref{eq: dual} with initial condition $\psi_0 \in c_{00} (\mathbb N^M )$ with $ \sup_{ n \in \mathbb N^M } \psi_0 (n) \leq 1 $. 
    Moreover 
   \begin{equation} \label{eq:bound for the dual at infinity}
   \psi(s, n ) \leq C e^{ - b  | M(n)  | e^{ - \gamma s }} \quad \forall n \in \mathbb N^M, \ s \in [0, T]
   \end{equation}
   where $\gamma \geq \max_{ k \in \Omega_c } \beta_k $ and 
   \begin{equation} \label{M(n)} 
   M(n) := ( m_j^T n )_{ j \in L } \in \mathbb N^L 
   \end{equation} 
   where $\mathcal B = \{ m_j \}_{j =1}^L $ is the extremal basis of conserved quantities, $0 < b < \dfrac{ 1} { \max_{ k \in \Omega_c } |M(e_k )| }  $ and $C \geq \max_{n \in \mathbb N^M } \left( \psi_0 (n) e^{ b |M(n) | } \right) $. 
 \end{proposition}  
In order to prove Proposition \ref{prop:existence of solution to the dual problem} is convenient to introduce the truncated dual problem 
\begin{align}  \label{eq:truncated dual}
 \partial_s  \psi^N (s, n)   = 
   \mathcal L_r^N[ \psi^N (s) ](n) + \overline {\mathcal L}_s^N  [ \psi^N  (s) ](n), \quad \psi(0, \cdot) = \psi_0, \quad n \in \bar I^M 
   \end{align}
where the operator $ \mathcal L^N_r $ is defined as in \eqref{truncated dual reaction} while 
\[
 \overline {\mathcal L}_s^N  [ \varphi (s) ](n) :=   \sum_{ k \in \Omega_c }  A^N_k (n)   \left[ \varphi(n+ e_k ) - \varphi(n)  \right] +    \sum_{ k \in \Omega_c } B_k^N (n)   \left[ \varphi(n-  e_k ) - \varphi(n)  \right]
\]
where 
  \begin{equation} \label{truncated sink}
    B_k^N  (n)  = \begin{cases} & \beta_k n_k, \text{ if } n \in \bar I^M, \\
    & 0 \ \text{ otherwise }. 
    \end{cases}
 \end{equation}
 In order to prove the existence of a solution to the dual problem it is convenient to truncate also the sink term because this guarantees that the truncated solution has compact support. 
 Indeed by the definition of the truncations we have that if $ \psi^N $ satisfies \eqref{eq:truncated dual} then $ \operatorname{ supp } \psi^N (t, \cdot) =\bar I^M $. 
 The existence of a solution to the truncated problem then can be proved using the classical existence results for systems of ODEs. 
Before proving that we notice that the truncated equation \eqref{eq:truncated dual} satisfies the maximum principle, i.e. the following statements hold. 
\begin{lemma} \label{lem:maximum principle trunc}
   Assume that the sequence of functions  $\psi^N (t)= (\psi_N(t, n ))_{n \in \mathbb N^M }  $ that are such that for every $n \in \mathbb N^M $ the function  $ t \mapsto \psi_N(t, n )$ is continuously differentiable and satisfies the following equation 
        \begin{align} \label{subsolution trunc}
 \partial_s  \psi^N(s, n)   -  
   \mathcal L^N_r[ \psi^N (s) ](n) - \bar{ \mathcal L}_s^N [ \psi^N (s) ](n) \leq 0 , \quad s \geq 0  \quad n \in \mathbb N^M     
\end{align}
    with initial datum $ \psi (0) = \psi_0 \in c_{00} (\mathbb N^M)$.
    Then 
    \begin{align*}
       \sup_{ t \in [0, \infty) } \max_{ n \in \mathbb N^M } \psi^N (t, n ) \leq    \sup_{ n \in \mathbb N^M } \psi_0 ( n ). 
    \end{align*}
    Moreover if $\psi^N $ attains a maximum at time $ t= \overline t  $, then $ \psi^N $ is constant. 
    
    Assume that the sequence of functions  $\psi^N (t)= (\psi_N(t, n ))_{n \in \mathbb N^M }  $ that are such that for every $n \in \mathbb N^M $ the function  $ t \mapsto \psi_N(t, n )$ is continuously differentiable and satisfies
        \begin{align} \label{supersolution trunc}
 \partial_s  \psi^N(s, n)   -  
   \mathcal L_r[ \psi^N (s) ](n) -  \mathcal L_s [ \psi^N (s) ](n) \geq 0 , \quad s\geq 0, \quad n \in \mathbb N^M     . 
\end{align}
    with initial datum $ \psi (0) = \psi_0 \in c_{00} (\mathbb N^M)$.
    Then 
    \begin{align*}
  \inf_{ t \in [0, \infty)   } \min_{ n \in \mathbb N^M } \psi^N (t, n ) \geq   \sup_{ n \in \mathbb N^M } \inf_0 ( n ) . 
    \end{align*}
    Moreover if $\psi^N $ attains a minimum at time $ t= \overline t  $, then $ \psi^N $ is constant. 
\end{lemma}
\begin{proof}
Since $\psi^N (t, \cdot) $ has compact support we have that $\max_{ n \in \bar{I}^M }  \psi^N (t, n ) =   \psi^N (t, \overline n ) $ for some $ \overline n \in \mathbb N^M $. We define the function $M(t) := \psi^N (t, \overline n )$ for $ t \geq 0. $
Since $ \psi^N $ is a solution to \eqref{eq:truncated dual} then we have that 
\begin{align*}
    \frac{d M (t) }{ dt } = \partial_t f(t, \overline n )  \leq 0. 
\end{align*}
As a consequence we have that $ M(t) \leq M(0) $ and the desired conclusion follows. The same argument can be adapted to prove the statement for the subsolutions \eqref{subsolution trunc}. 
\end{proof}
A direct consequence of Lemma \ref{lem:maximum principle} is the following lemma. 
\begin{lemma}[Comparison principle for the truncated problem] \label{lem:comparison principle truncated}
    Assume that $\psi^N  $  is a solution to \eqref{eq:truncated dual} with initial datum $ \psi^N (0) = \psi_0 $. Assume that $ \overline \psi^N   $  is a supersolution to \eqref{eq:truncated dual} with initial datum $ \overline \psi^N (0) = \overline \psi_0 \geq \psi_0  $, i.e.  it satisfies \eqref{supersolution trunc}
    Then $  \overline \psi^N (t, n )  \geq  \psi^N (t, n )  $ for every $ t >0 $, $ n \in \mathbb N^M $. 
\end{lemma}

We are ready to prove the existence of a solution to the dual truncated problem. 
 \begin{proposition}\label{prop:existence of solution to the dual problem}
Assume that $(\Omega, \mathcal R, \mathcal R_s, \mathcal K , \mathcal K_s) $ is a mass action chemical system with sources and sinks and that the set of reactions $\mathcal R $ is conservative.
    For any  $T>0 $ there exists  sequence of functions  $\psi^N (t)= (\psi_N(t, n ))_{n \in \mathbb N^M }  $ that are such that for every $n \in \mathbb N^M $ the function  $ t \mapsto \psi_N(t, n )$ is continuously differentiable and satisfies  equation \eqref{eq:truncated dual} with initial condition $\psi_0 \in c_{00} (\mathbb N^M )$ with $ \sup_{ n \in \mathbb N^M } \psi_0 (n) \leq 1 $.
 \end{proposition} 
 \begin{proof}
    In order to prove this Lemma we consider the truncated problem \eqref{eq:truncated dual}. 
The existence of a solution to the truncated equation can be proven using Picard-Lindelöf Theorem.
Indeed, equation \eqref{eq:truncated dual} can be written as 
\[
\frac{d \psi^N(t, n ) }{dt} = \mathcal T_n ( \psi^N ) , \quad  n \in I^M 
\]
where for every $n \in I^M $ the operator $ \mathcal T_n : \mathbb R_+^{  I^M} \rightarrow \mathbb R_*  $ is a locally Lipschitz function defined as 
\[ 
 \mathcal T_n (x): =  \mathcal L_r^N[ x  ](n) + \overline {\mathcal L}_s^N  [ x ](n). 
   \]
As a consequence there exists a time $T_*>0  $ such that there exists a solution to \eqref{eq:truncated dual} exists on $ [0, T_*] $. 
Notice moreover that $\psi^N (t, n ) \equiv 1 $ for every $ t \geq 0 $ and every $n \in \bar{ I}^M $ is a solution to \eqref{eq:truncated dual}.
We then can extend the solution to all positive times. 
 \end{proof}

We now prove the existence of a solution to \eqref{eq: dual} and the double exponential bound \eqref{eq:bound for the dual at infinity}.

\begin{proof}[Proof of Proposition \ref{prop:existence of solution to the dual problem}]
In order to prove the statement we aim at proving that there exists a limit $ \psi $ as $N \to \infty $ to the sequence of solutions to the truncated problem $\{ \psi^N \}_{ N \geq 1 } $. To this end we proceed in two steps. First of all we prove that the function $F(s,n ) := e^{ - b M (n) e^{ - \gamma s }} $ is a supersolution to \eqref{eq:truncated dual}. 
As a second step we will use this bound to prove the existence of a limit for the sequence $ \psi_N $ as $N \to  \infty$.

\textit{Step 1. Proof of the bound \eqref{eq:bound for the dual at infinity}. }
We first prove that the  function $F(s,n ) := e^{ - b M (n) e^{ - \gamma s }} $ is a supersolution to \eqref{eq:truncated dual}, i.e. is satisfies 
\[ 
\partial_s F (s, n ) - \overline{\mathcal L}_s^N [F] (s, n ) - \mathcal L_r^N [F] (s, n )  \geq 0 , \quad n \in \mathbb N^M , s \geq 0.  
\]
In order to prove this notice first of all that by the definition of $M $ we have that for every $\rho \in \{ 1, \dots, r \} $ it holds that 
\[ 
|M( n + S_\rho )| = \sum_{ j \in L } m_j^T (n+ S_\rho ) = \sum_{ j \in L } m_j^T n = |M(n )| \quad \text{ and } \quad  |M(n - S_\rho )| = \sum_{ j \in L } m_j^T (n-  S_\rho ) = |M(n)|. 
\]
Therefore, due to the definition of the operator $\mathcal L_r^N  $ we have that 
\begin{align*}
\mathcal L_r^N [F] (s, n ) = 
\mathcal L^N_r (e^{ - b | M (n) | e^{ - \gamma s }} )  = 0. 
\end{align*}
Moreover notice that the function $ n \mapsto F (s, n ) $ is decreasing for every $ s \in [0, T ] $. This is a consequence of the fact that the vectors $\{ m_j \}_{ j = 1 }^L $ of the extremal basis of conserved quantities are non-negative. 
As a consequence we deduce that 
\begin{align*}
   \overline{ \mathcal L}_s^N [F] (s, n )  &=      \sum_{ k \in \Omega_c }  A^N_k (n)   \left[ F(s, n+ e_k ) - F(s, n)  \right] +    \sum_{ k \in \Omega_c } B^N _k (n)   \left[ F(s, n-  e_k ) - F(s, n )   \right] \\
    & \leq \sum_{ k \in \Omega_c } B_k (n)   \left[ F(s, n-  e_k ) - F(s, n )   \right] 
\end{align*}
This, together with the definition of $F$ and of the constants $ \gamma $ and $b$, allows us to deduce that 
\begin{align*}
\partial_s F (s, n ) - \overline{\mathcal L}_s^N [F] (s, n ) - \mathcal L_r^N [F] (s, n )  & \geq \partial_s F (s, n ) - \sum_{ k \in \Omega_c } \beta_k n_k  \left[ F(s, n-  e_k ) - F(s, n )   \right] \\
& = b \gamma  |M (n) | e^{ - \gamma s } F(s,  n ) - \max_{ k \in \Omega_c } \beta_k \sum_{ k \in \Omega_c }  n_k  \left[ F(s, n-  e_k )  -  F(s, n )    \right] \\
& =  b \gamma  |M (n) | e^{ - \gamma s } F(s,  n ) -  \max_{ k \in \Omega_c } \beta_k  F(s, n ) \sum_{ k \in \Omega_c }  n_k  \left[ e^{ b |M (e_k )| e^{ - \gamma s } } -  1 \right]  \\
& \geq  b | M (n) |  e^{ - \gamma s } F(s,  n )  ( \gamma -  \max_{ k \in \Omega_c } \beta_k )  \geq 0.
\end{align*}
Therefore the function $F $ is a supersolution to \eqref{eq: dual}. 
We can apply to equation \eqref{eq:truncated dual} the maximum principle (Lemma \ref{lem:comparison principle truncated})  and  deduce that 
  \begin{equation} \label{exponential bound trunc}
  \psi^N (s, n ) \leq C e^{ - b  | M(n)  | e^{ - \gamma s }} \quad \forall n \in \mathbb N^M, \ s \in [0, T]
  \end{equation}
holds.

\textit{Step 2. Existence of a time dependent solution $ \psi$. }

As a consequence the uniform bound  $ \sup_{ s \in [0, T] } \max_{ n \in \mathbb N^M } \psi^N (s, n )  \leq 1 $ holds.
We deduce that the sequence  $\{  \psi^N \}_{ N \geq 1 } $ is uniformly bounded.
Moreover the bound \eqref{exponential bound trunc} can be used to prove the equicontinuity of the sequence  $\{  \psi^N \}_{ N \geq 1 } $. 
Indeed for every $ n \in \mathbb N^M $ it holds that 
\begin{align*}
\left| \psi^N(t_1, n )  - \psi^N (t_2, n ) \right| &= \int_{ t_1 }^{t_2 } \left| \mathcal L_s [ \varphi] (v,n) + \mathcal L_r [ \varphi](v,n) \right| dv \\
& \leq |t_2 - t_1 | \sup_{ v \in [0, T ] } \max_{ n \in \mathbb N^M} \left| \mathcal L_s [ \varphi] (v,n) + \mathcal L_r [ \varphi](v,n) \right| \\
& \leq C |t_2 - t_1 | \sup_{ v \in [0, T ] } \max_{ n \in \mathbb N^M } e^{ - b  | M(n)  | e^{ - \gamma v }}  \times \\
&\times \left[ \sum_{k \in \Omega_c } \left( A_k (n) + B_k (n) \right) +  \sum_{ \rho=1 }^r \left( R_\rho (n) + R_ { - \rho }(n) \right)  \right] \leq  \bar{C} |t_2 - t_1 |, 
\end{align*} 
for suitable constants $ C $ and $ \bar C $. 
Therefore, there exists a sequence of continuos functions $ \psi(t) =(\psi(t, n ))_{n \in \mathbb N^M} $ that is such that, up to a subsequence, for every $n \in \mathbb N^M $ we have that $ \psi^N (t, n ) \rightarrow \psi(t, n ) $ as $ N \to \infty $. The convergence is uniform for $ t $ in compact sets. 
We deduce that for every $n \in \mathbb N^M $ the function $  t \mapsto  \psi (t, n )   $ is continuously differentiable and satisfies \eqref{eq: dual}. 
Passing to the limit in all the terms of equation \eqref{eq:truncated dual} we deduce that the limit $\psi $ satisfies \eqref{eq:bound for the dual at infinity}. 
Moreover taking the limit as $N \to \infty $ in \eqref{exponential bound trunc} we deduce that the bound \eqref{eq:bound for the dual at infinity} holds. 
\end{proof}

 Using \eqref{eq:bound for the dual at infinity} we can now formulate the maximum principle for the dual equation \eqref{eq: dual}. 
\begin{lemma}[Maximum principle] \label{lem:maximum principle}
    Assume that $\psi (t)=(\psi (t, n ))_{ n \in \mathbb N^M} $ is a sequence of functions that is such that for every $n \in \mathbb N^M $ the function $ t \mapsto f(t, n ) $ is continuously differentiable and satisfies 
        \begin{align} \label{subsolution}
 \partial_s  \psi(s, n)   -  
   \mathcal L_r[ \psi (s) ](n) - \mathcal L_s [ \psi (s) ](n) \leq 0 , \quad s\geq 0,  \quad n \in \mathbb N^M   . 
\end{align}
Then we have that 
        \begin{align*}
       \sup_{ t \in [0, \infty ) } \sup_{ n \in \mathbb N^M } \psi (t, n ) \leq    \sup_{ n \in \mathbb N^M } \psi_0 ( n ). 
    \end{align*}
        Moreover if $\psi $ attains a maximum at time $ t= \overline t  $, then $ \psi $ is constant.  

          Assume that $\psi (t)=(\psi (t, n ))_{ n \in \mathbb N^M} $ is a sequence of functions that is such that for every $n \in \mathbb N^M $ the function $ t \mapsto f(t, n ) $ is continuously differentiable and satisfies 
        \begin{align} \label{supersolution}
 \partial_s  \psi(s, n)   -  
   \mathcal L_r[ \psi (s) ](n) - \mathcal L_s [ \psi (s) ](n) \geq 0 , \quad s\geq 0  \quad  n \in \mathbb N^M   . 
\end{align}
Then we have that 
        \begin{align*}
       \inf_{ t \in [0, \infty ) } \inf_{ n \in \mathbb N^M } \psi (t, n ) \geq    \inf_{ n \in \mathbb N^M }  \psi_0 ( n ). 
    \end{align*}
        Moreover if $\psi$ attains a minimum at time $ t= \overline t  $, then $ \psi $ is constant.  
\end{lemma}
\begin{proof}
    The bound \eqref{eq:bound for the dual at infinity} implies that for every $ t \geq 0 $ there exists $ \overline n \in \mathbb N^M $ such that $ \sup_{n \in \mathbb N^M } \psi (t, n ) =   \psi (t, \overline n )$. Using the argument used in order to prove Lemma \ref{lem:maximum principle trunc} we deduce the result. 
\end{proof}
    \begin{lemma}[Comparison principle] \label{lem:comparison principle}
    Assume that $ \overline  \psi  $  is a supersolution to \eqref{eq: dual}, i.e. it satisfies \eqref{supersolution} with initial datum $ \overline \psi (0) = \overline \psi_0 $. 
    Then $  \overline \psi (t, n )  \geq  \psi (t, n )  $ for every $ t >0 $, $ n \in \mathbb N^M $. 
    
    Assume that $ \underline  \psi $  is a subsolution to \eqref{eq: dual}, i.e. assume that it satisfies \eqref{supersolution} with initial datum $ \underline \psi (0) = \overline \psi_0 $. 
    Then $  \underline \psi (t, n )  \leq  \psi (t, n )  $ for every $ t >0 $, $ n \in \mathbb N^M $.  
\end{lemma}
An immediate consequence of the maximum principle is the uniqueness of solutions to the dual problem \eqref{eq: dual}. 
\begin{corollary}
Assume that $(\Omega, \mathcal R, \mathcal R_s, \mathcal K , \mathcal K_s) $ is a mass action  chemical system with sources and sinks and that the set of reactions $\mathcal R $ is conservative. Let $T>0 $, then there exists a unique sequence of functions $\psi (t)=(\psi( t , n ))_{ n \in \mathbb N^M} $ that is such that for every $ n \in \mathbb N^M $ the function  $ t \mapsto \psi(t, n )  $ is continuously differentiable and satisfies equation \eqref{eq: dual} with initial condition $\psi_0 \in c_{00} (\mathbb N^M )$.
\end{corollary}

We are now ready to prove the uniqueness of the time dependent solution to \eqref{eq:master with sources weak}. 
\begin{proof}[Proof of Theorem \ref{thm:existence and uniqueness of time dep}]
    The existence of a time dependent solution $f$ was proven in Theorem \ref{thm:existence of time dep} we now prove that the solution is unique. 
    Assume $f_1 $ and $f_2 $ to are two solutions to equation \eqref{eq:master with sources weak}     with the same initial datum $ f_1(0, \cdot) = f_2(0, \cdot )  $. 
    Then for every $ \varphi \in C^1((0, \infty) , c_{00}(\mathbb N^M ) $ it holds that 
       \begin{align*}
    \sum_{  n \in \mathbb  N^M  } \varphi(t,n) [ f_1 (t, n) - f_2 (t, n ) ]  &= \sum_{  n \in \mathbb  N^M  } \varphi(0, n) [  f_1 (0,  n) - f_2(0, n)]  \\
    & + \sum_{ n  \in \mathbb  N^M }  \int_0^t \left[ \mathcal L_r[ \varphi](n,s) +  \mathcal L_s [ \varphi](s,n ) + \partial_s \varphi (s, n )  \right]   [ f_1 (s,n) - f_2 (s, n ) ] ds. 
\end{align*}

       Assume that  $\psi $ is a solution to \eqref{eq: dual} with initial datum $ \psi_0 (n ) = e_{ \overline n } $ for some $ \overline n \in \mathbb N^M $. Proposition \ref{prop:existence of solution to the dual problem} guarantees that $ \psi $ satisfies the exponential bound \eqref{eq:bound for the dual at infinity}. 
Since  $\sup_{ n \in \mathbb N^M } f_k(t, n ) \leq 1 $ for $ k \in \{ 1, 2 \} $ and since the solution to the dual problem \eqref{eq: dual} satisfies the  bound \eqref{eq:bound for the dual at infinity} we can consider the test function $\varphi (t, n ) =  \psi(T- t  , n ) $ in equation \eqref{eq:master with sources weak}.

    We deduce that for every $ t \in [0, T] $ it holds that 
   \begin{align*}
    \sum_{  n \in \mathbb  N^M  } \varphi(t,n) [  f_1 (t, n) - f_2(t, n ) ] = \sum_{  n \in \mathbb  N^M  } \varphi(0, n) [  f_1 (0, n) - f_2(0, n ) ]=0 . 
\end{align*}
Evaluating at $t = T $ we deduce that 
   \begin{align*}
   0=  \sum_{  n \in \mathbb  N^M  } \varphi(T,n) \left[  f_1 (T, n)  - f_2 (T, n ) \right] =  \sum_{  n \in \mathbb  N^M  } \psi_0(n)  \left[  f_1 (t, n)  - f_2 (t, n ) \right] =  f_1 (t, \overline  n)  - f_2 (t, \overline n ). 
\end{align*}
Due to the fact that $ \overline n \in \mathbb N^M $ is arbitrary we conclude that $  f_1 (t,  n)  =  f_2 (t, n ) $ for every $n \in \mathbb N^M $. 
\end{proof}

\section{Stability of the stationary solution} \label{sec:convergence}
The goal of this section is first of all to prove that the master equation \eqref{eq:master with sources weak} corresponding to a chemical system that has conservative reactions has a unique stationary solution $\overline f $. Moreover we will also prove that the time dependent solution $f $ to the master equation \eqref{eq:master with sources weak} tends to the stationary solution $\overline f $ as $t \to \infty $. 
We start this section stating the definition of stationary solution. 
\begin{definition}[Stationary solution to the master equation] \label{def:stationary solution}
A steady state to \eqref{eq:weak form master trunc} is a $g \in \mathcal P ( \mathbb N^M ) $ that satisfies  
\begin{align} \label{eq:master with sources weak stationary}
0= \sum_{ n \in \mathbb N ^M }  g ( n)  \mathcal L_s [\varphi] ( n ) +  \sum_{ n \in \mathbb N ^M }   g ( n)  \mathcal L_r   [\varphi] (n ) 
\end{align}
for every $ \varphi \in c_{ 00}( \mathbb N^M)$. 
\end{definition}
\subsection{Existence of a stationary solution}
The goal of this section is to prove the existence of stationary solutions for chemical systems with conservative reactions. 
\begin{theorem} \label{thm:steady state}
    Assume that $(\Omega, \mathcal R, \mathcal R_s, \mathcal K , \mathcal K_s) $ is a mass action chemical system with sources and sinks and that the set of reactions $\mathcal R $ is conservative. 
    Then there exists a steady state $ f_s \in \mathcal P(\mathbb N^M ) $. 
\end{theorem}
\begin{proof}
In order to prove the result we aim at using Brouwer fixed point theorem.
In particular we use a type of argument that has been extensively used in order to prove the existence of self-similar solution to a class of coagulation equations, see for instance \cite{ferreira2022self}.
In particular, we will use inequality \eqref{invariant region} to define a compact convex invariant region. 
In order to use inequality \eqref{invariant region} to construct an invariant region it is convenient to prove the existence of a truncated stationary solution $g^N $, i.e. a solution to 
\begin{align} \label{eq:master with sources weak stationary}
0= \sum_{ n \in \mathbb N ^M }  g ( n)  \mathcal L^N_s [\varphi] ( n ) +  \sum_{ n \in \mathbb N ^M }   g ( n)  \mathcal L^N_r   [\varphi] (n ). 
\end{align}

    Consider the truncated master equation  \eqref{eq:master without sources truncated}. 
    Proposition \ref{prop:existence truncated} guarantees the existence of a solution $f^N  $ to equation \eqref{eq:master without sources truncated}. 
    We define the semigroup $ \{ S(t) \}_{ t \geq 0 } $ with values in  $ \mathcal P (\mathbb N^M )   $ that is defined as $S(t) f_0 = f^N (t, \cdot) $ for every $ t >0 $, where $f^N $ is the solution to \eqref{eq:master without sources truncated} with respect to the initial datum $ f_0 \in \mathcal P  (\mathbb N^M )$. We plan to use Brouwer fixed point theorem (cf. \cite{evans2022partial}) to prove that for every $t >0 $ there exists  a $ \overline f_t \in \mathcal P (\mathbb N^M ) $ that satisfies $\overline f_t $ satisfying $S(t) \overline f_t = \overline f_t $. 
    To this end we first of all prove that the operator $S(t) $ is continuous in the $\ast$-weak  topology. To this end we use a similar approach to the one that we used to prove the uniqueness of the solution $f$ to \eqref{eq:master with sources weak}. 
    
    Consider the solution $f^N_1 $ to equation \eqref{eq:master with sources weak} with initial condition $f^1_0$ and the solution $f^N_2$ with initial condition $f^2_0$. Assume that the two initial conditions $f^1_0 $ and $f^2_0 $ are close in the $\ast$-weak topology. 
The goal is to construct a test function $ \varphi $ that is such that $ \varphi (T) = \overline \varphi $ has compact support and
\[ 
\sum_{ n \in \mathbb N^M }  \overline \varphi ( n ) \left(  f^N_1 (T, n ) - f^N_2 (T, n ) \right) = \sum_{ n \in \mathbb N^M }  \varphi_0 ( n ) \left(  f_0^1 ( n ) - f_0^2 ( n ) \right). 
\] 
To this end recall that for every final datum $ \varphi (T) = \overline \varphi \in c_{ 00 } (\mathbb N^M)   $ the equation $ \partial_t  \varphi + \mathcal L^N [ \varphi ] =0 $ has a solution  $ \varphi (t, n) $ with support in $ n \in \bar{I} ^M $.

Moreover the solution $\varphi$ satisfies the exponential bound $\varphi (t, n ) \leq C e^{ - b M(n) } e^{- \gamma t } $. 
Therefore we can consider the solution $ \varphi $ as a test function in \eqref{eq:master with sources weak} and obtain that 
\begin{align*} 
\sum_{  n \in \mathbb  N^M  } \overline \varphi(n) \left(  f_N^1  (T , n) - f_N^2 (T, n ) \right)  &= \sum_{  n \in \mathbb  N^M  } \varphi_0( n)  \left(  f^1_0  (n) - f^2_0 ( n ) \right) \\
& + \sum_{ n  \in \mathbb  N^M }  \int_0^t \left[ \mathcal L^N_r[ \varphi](n,s) + \bar{ \mathcal L}^N_s [ \varphi](s,n ) - \partial_s \varphi (s, n )  \right]  \left(  f^N_1  (s, n) - f^N_2 (s, n ) \right)   ds  \\ \nonumber
& = \sum_{  n \in \mathbb  N^M  } \varphi_0( n)  \left(  f^1_0  (n) - f^2_0 ( n ) \right) \\
& + \sum_{ n  \in \bar{I}^M }  \int_0^t \left[ \mathcal L^N_r[ \varphi](n,s) +  \mathcal L^N_s [ \varphi](s,n ) - \partial_s \varphi (s, n )  \right]  \left(  f^N_1  (s, n) - f^N_2 (s, n ) \right)   ds  \\
    & = \sum_{  n \in \mathbb  N^M  } \varphi_0( n)  \left(  f^1_0  (n) - f^2_0 ( n ) \right). 
    \end{align*}  
    We can then use the fact that $\varphi_0 (n) \leq C e^{ -  b M(n) } $ to approximate $\varphi_0$ with test functions that have compact support and to deduce that for every $ \overline \varphi $
    \[
    \sum_{  n \in \mathbb  N^M  } \overline \varphi(n) \left(  f^N_1  (T , n) - f_N^2 (T, n ) \right)  \to 0 \text { as } f^0_1 \rightharpoonup^\ast f^0_2. 
    \]
    This proves the the $\ast$-weak continuity of the operator $S(t)$. 
    
    We now prove that the operator has a compact convex invariant region. 
    Indeed, inequality \eqref{invariant region} guarantees that the set
    \[ 
    \chi_{\gamma \delta} := \left\{ h  \in  \mathcal P (\mathbb N^M )  : \sum_{ n \in \mathbb N^M } m^T n h (n) \leq \dfrac{\gamma }{\delta} \right\},  
    \]
     where $\gamma :=  |\Omega_c| \max_{k \in \Omega_c }  \alpha_k ( m^T e_k ) >0 $ and $ \delta = \min_{\Omega_c } \beta_k >0  $, is an  invariant region for $S(t)$. 
     Moreover the set $\chi_{\gamma\delta}   $ is  a convex set and is compact in the $\ast$-weak topology of $\ell^1 (\mathbb N^M ) $. 
     We then apply Brouwer fixed point theorem in order to prove that there exists a $ \overline f_t \in \mathcal P (\mathbb N^M ) $ that satisfies $S(t) \overline f_t = \overline f_t $.

    In order to conclude the proof we need to show that for every $ f_0 $ the map $t \mapsto S(t) f_0 $ is $\ast-$weak continuous. 
    This follows by the fact that for every $\varphi \in c_{ 00} ( \mathbb N^M ) $ we have that
    \begin{align*}
        \sum_{ n \in \mathbb N^M }  \varphi (n) \left[  f^N (t_1, n ) - f^N (t_2, n ) \right]  & = \sum_{ n \in \mathbb N^M } \left(\mathcal L_r^N [\varphi] (  n ) + \mathcal L_s^N [\varphi] ( n ) \right) \int_{ t_1 }^{t_2 }  f^N (s, n )  ds \\
        & \leq C_\varphi (t_2- t_1)  
        \end{align*}
    for every $ t_2 > t_1 $ and for a constant $C_\varphi>0 $ that depends on $\varphi $. 
Combining the $\ast-$weak continuity of $ t \mapsto f(t) f_0 $ and the existence of a fixed point  $S(t) \overline f_t = \overline f_t $  for every time we deduce that there exists a $ g^N \in \mathcal P (\mathbb N^M ) $ that is such that $ S(t) g^N = g^N $ for every $ t >0 $. 

We now aim at taking the limit as $N \to \infty$. To this end notice that the sequence $ \{ g^N \}_{ N \geq 1 } \subset \chi_{\gamma \delta} $ where $ \chi_{\gamma \delta} $ is compact in the  $\ast-$weak topology of $ \ell^1 ( \mathbb N^M ) $. As a consequence, up to selecting a subsequence, we have that there exists $ g \in  \ell^1 ( \mathbb N^M ) $ such that 
\[ 
g^N   \rightharpoonup^\ast  g  \  \text{ as  } N \to \infty. 
\]
In particular this implies that 
\begin{align*}
\lim_{N \to \infty }     \sum_{ n \in \mathbb N ^M }  g^N  ( n)  \mathcal L_s [\varphi] ( n )=  \sum_{ n \in \mathbb N ^M }  g  ( n)  \mathcal L_s [\varphi] ( n ) 
    \end{align*} 
    as well as 
    \begin{align*} 
 \lim_{N \to \infty }     \sum_{ n \in \mathbb N ^M }   g^N ( n)  \mathcal L_r   [\varphi] (n )  = \sum_{ n \in \mathbb N ^M }  g  ( n)  \mathcal L_r [\varphi] ( n ). 
\end{align*}
As a consequence, since it also holds that 
\begin{align*}
\lim_{N \to \infty }     \sum_{ n \in \mathbb N ^M }  g  ( n)  \left[ \mathcal L_s [\varphi] ( n ) - \mathcal L^N_s [\varphi] ( n ) \right]= 0 \text{ and } \lim_{N \to \infty }     \sum_{ n \in \mathbb N ^M }  g  ( n)  \left[ \mathcal L_r [\varphi] ( n ) - \mathcal L^N_r [\varphi] ( n ) \right]= 0
    \end{align*}
we deduce  that $g $ is a solution to \eqref{eq:master with sources weak}. 
The fact that $ \sum_{ n \in \mathbb N^M} g (n) =1 $ follows by the fact that $ \sum_{ n \in \mathbb N^M} g_N (n) =1 $  for every $N \geq 1$ and by the fact that $g $ belongs to the set $\chi_{\gamma \delta }  $, hence $ \sum_{n \in \mathbb N^M } |n| g(n) < C_0 $ where $C_0>0$ does not depend on time. 
\end{proof}  

\subsection{Convergence to the steady state}

The aim of this section is to prove that the stationary solution $\overline f$  to the master equation \eqref{eq:master with sources weak} is unique and, moreover, for any initial datum the time dependent solution to the master equation \eqref{eq:master with sources weak} converges to the unique stationary solution $\overline f $ as time tends to infinity, i.e. we aim at proving Theorem \ref{thm: convergence to steady state}. 
 
To this end, it is convenient to find a suitable subsolution to equation \eqref{eq: dual}. 
\begin{proposition} \label{prop:subsolution dual}
Let $ \varepsilon >0 $ and $c_0 \geq 1 $. 
Let us consider the function $\Psi_\varepsilon $ defined as 
\begin{equation} \label{psi epsilon}
\Psi_\varepsilon (t, n ) := \lambda (t) G (t,  \varepsilon |M(n)| ) 
\end{equation}
where $M (n) $ is defined as in \eqref{M(n)} and where the function $ \lambda: [ 0, \infty ) \mapsto [ 0 , \infty) $ is defined as 
\[ 
\lambda (t) := c_0 \exp \left( - \dfrac{ \varepsilon c (1- e^{ - \delta t } ) }{\delta  }  \right)  \ \text{ and }\ 
G(t, \xi ) := \dfrac{ \exp \left( - a \xi e^{- \delta t } \right)  }{ 1+\exp \left( - a \xi e^{- \delta t } \right) }
\]
where $a \geq 0$,  $ \delta \geq |\Omega_c | \left( \max_{ k \in \Omega_c } M(e_k) \right)^2  $ and  $c\geq \dfrac{a |\Omega_c | \min_{ k \in \Omega_c } ( \alpha_k M(e_k )) }{2}$. 
Then $\Psi_\varepsilon $ is a subsolution to equation \eqref{eq: dual}, i.e. it satisfies \eqref{subsolution}. 
\end{proposition}
\begin{proof}
In order to prove the result we first of all notice that the definition of $M $ guarantees that 
\[ 
\mathcal L_r [ \Psi_\varepsilon] (t, n ) =0. 
\]
Moreover we have that for every $t >0 $ and $n \in \mathbb N^M $ it holds that
\begin{align*} 
\mathcal L_s [ \Psi_\varepsilon] (t, n ) = \sum_{ k \in \Omega_c }  \alpha_k \lambda(t) \left[ G(\varepsilon |M ( n + e_k ) |-  G(\varepsilon |M ( n  ) | ) \right] + \sum_{ k \in \Omega_c }  \beta_k n_k  \lambda(t) \left[ G(\varepsilon |M ( n - e_k ) |) -  G(\varepsilon |M ( n  ) |  \right]. 
\end{align*} 
Using Taylor's theorem  and the change of variables $ \xi = \varepsilon | M(n) |$ we obtain that
\[ 
 G(\varepsilon| M ( n + e_k )| ) -  G(\varepsilon |M ( n  ) |) ) =  G( \xi + \varepsilon |  M(e_k )| ) -  G(\xi  )  = \varepsilon |M(e_k)| \partial_\xi G(t, \xi ) + R_\varepsilon^k  (t, \xi) 
\]
and 
\[ 
 G(\varepsilon |M ( n - e_k ) |) -  G(\varepsilon |M ( n  ) | ) ) =  G( \xi -  \varepsilon |  M(e_k ) | ) -  G(\xi  )  =-  \varepsilon |M(e_k)| \partial_\xi G(t, \xi ) + R_\varepsilon^k  (t, \xi) 
\]
where 
\[ 
R_\varepsilon^k   (t, \xi) :=\frac{1}{2} \int_{ \varepsilon M(e_k) }^{ \varepsilon M(e_k) +\xi  } \partial_{\xi\xi} G(t, \xi ) \left( \xi + \varepsilon | M (e_k )| - s  \right)^2 ds  
\]
Therefore 
\begin{align*} 
\mathcal L_s [ \Psi_\varepsilon] (t, n ) = \sum_{ k \in \Omega_c }  ( \alpha_k - \beta_k n_k ) \varepsilon | M(e_k)| \lambda(t)  \partial_\xi G(t, \xi ) + \sum_{ k \in \Omega_c }  ( \alpha_k + \beta_k n_k )  R_\varepsilon^k  (t, \xi). 
\end{align*} 
We now notice that the function $G$ is decreasing and convex. Indeed we have that 
\[ 
\partial_\xi G(t, \xi ) = - \dfrac{  a  e^{ - \delta t }  \exp  \left(  a \xi e^{  - \delta t } \right)  }{ (1+\exp \left(  a \xi e^{ -  \delta t } \right))^2  } = -  \dfrac{ a  e^{ - \delta t }     }{ \left( \cosh \left( \dfrac{ a \xi e^{ - \delta t } }{2} \right)\right)^2   } \leq 0, \quad t >0,\  \xi >0 
\]
and 
\[ 
\partial_{\xi\xi} G(t, \xi ) = \dfrac{ a^2  e^{ -2 \delta t }  \sinh \left( \dfrac{ a \xi e^{ - \delta t } }{2} \right)  }{ 2 \left( \cosh \left( \dfrac{ a \xi e^{ - \delta t } }{2} \right)\right)^3 }  \geq 0 \quad t >0, \ \xi >0 . 
\]
As a consequence we have that for every $ k \in \Omega_c $ it holds that 
\[ 
R_\varepsilon^k   (t, \xi)  \geq 0, \quad \forall t \geq 0 , \  \xi \in \mathbb R_+ 
\]
In particular this implies that 
\begin{align*} 
\mathcal L_s [ \Psi_\varepsilon] (t, n ) \geq  \sum_{ k \in \Omega_c }  ( \alpha_k - \beta_k n_k ) \varepsilon | M(e_k)| \lambda(t)  \partial_\xi G(t, \xi ). 
\end{align*}  
Since $\xi = \varepsilon M (n) = \varepsilon \sum_{j =1}^L m_j^T n  $ we know that there exists a constant $ C>0 $ such that $n_k \leq   \dfrac{\xi C  }{ \varepsilon } $ for every $ k \in \Omega_c $. Then we have that 
\begin{align*} 
\mathcal L_s [ \Psi_\varepsilon] (t, n ) \geq  \sum_{ k \in \Omega_c }  ( \alpha_k \varepsilon - C \xi  ) | M(e_k)| \lambda(t)  \partial_\xi G(t, \xi ) \geq \lambda(t)   \left( A \varepsilon  - \delta  \xi \right) \partial_\xi G(t, \xi ) 
\end{align*} 
where $A : = |\Omega_c | \min_{ k \in \Omega_c } ( \alpha_k |M(e_k )|)$ and where $\delta := C \max_{ k \in \Omega_c } | M (e_k) |$. 

This allows us to deduce that 
\begin{align*} 
\partial_t  \Psi_\varepsilon (t, n )  - \mathcal L_s [ \Psi_\varepsilon] (t, n ) &\leq \partial_t  \Psi_\varepsilon (t, n ) -  \lambda(t)   \left( A \varepsilon  - \delta  \xi \right) \partial_\xi G(t, \xi )  \\
& = \dot{\lambda (t) } G(t, \xi ) +\lambda (t) \partial_\xi G(t, \xi ) -  \lambda(t)   \left( A \varepsilon  - \delta  \xi \right) \partial_\xi G(t, \xi ). 
\end{align*}
By the definition of $ \lambda $, of $G$, of $A$ and of $c$ we know that 
\begin{align*} 
 \dot{\lambda (t) } G(t, \xi )  - A \varepsilon \lambda (t) \partial_\xi G(t, \xi ) = G   (t, \xi )   \lambda (t)  \varepsilon e^{ - \delta t }  \left( \dfrac{ A a   } {1+ \exp (- a \xi e^{- \delta t }) } - c \right) \leq 0. 
\end{align*}
Therefore 
\begin{align*} 
\partial_t  \Psi_\varepsilon (t, n )  - \mathcal L_s [ \Psi_\varepsilon] (t, n ) &\leq \lambda (t) ( 1+ \delta \xi )  \partial_\xi G(t, \xi  ) \leq 0. 
\end{align*}
We deduce that $ \Psi_\varepsilon $ is a subsolution to \eqref{eq: dual}. 
\end{proof}
The subsolution constructed in Proposition \ref{prop:subsolution dual} can be used in order to find a lower bounds to solutions to \eqref{eq: dual}  that have as initial conditions Heaviside functions. 
\begin{lemma} \label{cor:subsolution in R}
Let $R>0 $. 
Assume that $ H $ is the solution to equation \eqref{eq: dual} with initial condition 
\[ 
H_0(n) = \mathds{1}_{|M(n) | \leq R  } (n).  
\] 
Then for every $n \in \mathbb N^M $ it holds that 
\begin{equation} \label{eq:lower bound subsolution}
\liminf_{ t \to \infty }   H(t, n )  \geq    1-  \dfrac{c \varepsilon_R}{\delta } \left( 1+ c_a \right)
\end{equation}
where  $ a, c, \delta  $ are as in Proposition \ref{prop:subsolution dual} and 
\[ 
\varepsilon_R := \dfrac{1}{a R } \ln \left( 1+ \dfrac{2}{c_a } \right)
\]
for some $ c_a >0 $. 
\end{lemma}
\begin{proof}
Let us define $c_0 := 2 (1+ c_a ). $
Consider the function $ \Psi^+_{\varepsilon_R } (t, n ) $ defined as 
\[ 
\Psi^+_{\varepsilon_R } (t, n ):= \left(\Psi_{\varepsilon_R } (t, n ) - c_a \right)_+ 
\]
where $ \Psi_{\varepsilon_R }$ is given by \eqref{psi epsilon}. 

 By construction we have that  $ H_0(n) = \mathds{1}_{|M(n) | \leq R  } (n) \geq \Psi^+_{\varepsilon_R } (0, n )$ for every $n \in \mathbb N^M$.
 Since the function $\Psi^+_{\varepsilon_R } (t, n )$ is a subsolution we deduce that for every $ n \in \mathbb N^M$ 
 \[ 
 \liminf_{ t \to \infty }   H(t, n )  \geq \liminf_{ t \to \infty } \Psi^+_{\varepsilon_R } (t, n ).  
 \] 
 Using the definition of $\Psi^+_{\varepsilon_R } (t, n )$ 
 we obtain that 
 \[
  \liminf_{ t \to \infty } \Psi^+_{\varepsilon_R } (t, n ) \geq ( 1+c_a) \exp \left( - \dfrac{\varepsilon_R c   }{\delta }\right) - c_a \geq   1-  \dfrac{c \varepsilon_R}{\delta } \left( 1+ c_a \right). 
 \]
 Hence inequality holds \eqref{eq:lower bound subsolution}. 
\end{proof}
Proposition \ref{prop:subsolution dual} and Lemma \ref{cor:subsolution in R} allow us to prove that the solution to the dual problem \eqref{eq: dual} converges to a constant $ \psi_\infty $ as time tends to infinity. 
\begin{proposition}
Assume that $ \psi_0 \in c_{ 00 } ( \mathbb N^M )  $ is such that $\|  \psi \|_\infty \leq 1  $ and such that $ \psi_0 (n) \geq 0 $ for every $n \in \mathbb N^M $. Then the unique solution $ \psi $ to the dual problem \eqref{eq: dual} is such that there exists a positive  $ \psi_\infty \in \mathbb R $ such that 
\[ 
\lim_{ t \to \infty} \psi(t, n ) = \psi_\infty , \quad \forall n \in \mathbb N^M. 
\]    
\end{proposition}
\begin{proof}
In order to prove the statement it is convenient to use the notation $ \underline \psi ( n ) = \liminf_{ t \to \infty } \psi(t, n )  $ and $ \overline \psi ( n ) = \limsup _{ t \to \infty } \psi(t, n )  $  for all $n \in \mathbb N^M $. 
Moreover, we define $\underline \psi_\infty $ and $\overline  \psi_\infty $ as
\[
\underline \psi_\infty := \inf_{ n \in \mathbb N^M }  \underline \psi(n) \quad  \text{ and } \quad  \overline \psi_\infty := \sup_{ n \in \mathbb N^M }  \overline \psi(n). 
\]
We aim at proving that $\overline \psi_\infty = \underline \psi_\infty = \psi_\infty $. To this end we argue by contradiction, hence we assume that 
\[ 
\overline \psi_\infty > \underline  \psi_\infty. 
\]
By the definition of supremum we know that for every $ \bar \varepsilon >0 $ there exists a $ n_0 \in \mathbb N^M $ that is such that 
\[ 
\overline \psi (n_0 ) \geq \overline \psi_\infty - \bar \varepsilon. 
\]
In particular we select $ \bar \varepsilon := \dfrac{\overline \psi_\infty  -  \underline \psi_\infty}{3} $. As a consequence $n_0 $ depends only on $\overline \psi_\infty $ and $  \underline \psi_\infty$. 

Consider $R >0 $ sufficiently large in order to have that $ |M(n_0) | < R $. Then by the definitions of $\underline \psi_\infty$ and $\overline \psi_\infty $ we know that 
 there exists a time $t_{ \overline  \varepsilon R } >0 $ that is such that 
 \[ 
\psi (t_{ \bar \varepsilon R }  , n_0  ) \geq \overline \psi (n_0) - \varepsilon \geq \overline  \psi_\infty - 2 \bar \varepsilon >0.
 \]
On the other hand we by the definition of infimum  it holds that 
\[ 
\liminf_{ t \to \infty } \psi (t , n  ) \geq \underline \psi_\infty. 
\]
Moreover, by the definition of liminf we have that  for every $ \varepsilon >0 $ there exists a sufficiently large $T>0 $ such that for  $t >  T>0 $
\[ 
\psi(t, n ) \geq \liminf_{ t \to \infty } \psi (t , n  ) \geq \underline \psi_\infty - \varepsilon \quad  \forall n \ \text{ s.t. } \ |M(n) | \leq R. 
\]

 In order to prove the last inequality we used the fact that the set $ B_R := \{  n \in \mathbb N^M : | M(n) |\leq R \} $ is a finite set. 
 Consider now the solution $H $ to the dual equation $ \partial_t H = \mathcal L_s [H] + \mathcal L_r [H] $ that is such that 
 \[ 
 H( \bar t, n ) = \begin{cases}
      \overline \psi_\infty - 2 \bar \varepsilon  & \text{ if } n = n_0 \\
     (  \underline \psi_\infty- \varepsilon) \mathds{1}_{B_R} (n)  & \text{ if } n \neq n_0   
 \end{cases}
 \]
 for $ \bar t \geq \max \{  T, t_{ \bar \varepsilon R } \}$.
 Notice that by definition it holds that $ \psi (\bar t , n )  \geq H ( \bar t , n )   $ for every $n \in \mathbb N^M $. 
 We can apply the comparison principle and deduce that $ \psi (t, n )  \geq H (t , n )   $ for every $n \in \mathbb N^M$ and for every $t \geq \bar t $. We now aim at finding a bound from below for the function $H$. 
 Let us define $H_1 $ to be the solution to \eqref{eq: dual} with initial condition 
 \[
H_1( \bar t, n ) = \begin{cases}
     \overline \psi_\infty - 2 \bar \varepsilon +   \varepsilon - \underline \psi_\infty & \text{ if } n = n_0 \\
      0 & \text{ if } n \neq  n_0
\end{cases}
 \]
 with $H_2 $ the solution to  \eqref{eq: dual} with initial condition $H_2 ( \bar t, n ) =  \underline \psi_\infty- \varepsilon  $ for every $n \in \mathbb N^M$. 
 Finally let us denote with $H_3 $ the solution to \eqref{eq: dual} with initial datum $ H_3( \bar t, n )  =  - (  \underline \psi_\infty- \varepsilon) \mathds{1}_{ n \in \mathbb N^M \setminus {B_R}} $. 
 Notice that by construction $H = H_1 + H_2 + H_3$. 
Moreover $H_2(t, n ) =\underline \psi_\infty- \varepsilon $ for every $ n \in \mathbb N^M $ and every $ t \geq \bar t $. 

 Assume now that $ 0 < \overline R < R  $ is such that $n _0 \in B_{ \overline R } $. 
Due to the strong maximum principle (Lemma \ref{lem:maximum principle}) we have that for every $ n \in B_{\overline R } $
\[ 
H_1( \bar t +1 , n )   \geq \min_{ n \in B_{\overline R} } H_1( \bar t +1 , n )   >0.  
\] 
As a consequence, there exists a constant  $\nu_{\overline R } >0 $ that depends on $ \overline R $ and that is such that 
\[ 
H_1(\bar t +1 , n )   \geq \nu_{\bar R  } ( \overline \psi_\infty - 2 \bar \varepsilon +  \varepsilon - \underline \psi_\infty ), \quad  n \in B_{\overline R }  .  
\] 
Summarizing we proved that 
\[ 
H_1(\bar t +1 , n ) + H_2 (\bar t +1 , n )   \geq  \underline \psi_\infty- \varepsilon + \nu_{\bar R  }( \overline \psi_\infty - 2 \bar \varepsilon +  \varepsilon - \underline \psi_\infty ), \quad  n \in B_{\overline R }.
\]
We now consider the function $F$ that satisfies \eqref{eq: dual} and that is such that 
\[ 
F( \bar t  +1 , n ) = \underline \psi_\infty- \varepsilon + \nu_{\bar R  } ( \overline \psi_\infty - 2 \bar \varepsilon +  \varepsilon - \underline \psi_\infty )  \mathds{1}_{ B_{ \overline R }} (n). 
\]
By the comparison principle we know that $ H(t, n ) \geq F(t, n ) $ for every $ n \in \mathbb  N^M $ and every $ t \geq \bar t  +1 $. 
Moreover we have that 
\begin{align} \label{ineq for F}
F(t, n ) \geq   \underline \psi_\infty- \varepsilon + \nu_{\bar R  } ( \overline \psi_\infty - 2 \bar \varepsilon +  \varepsilon - \underline \psi_\infty ) G  (t - \bar t -1, n )  , \quad  \forall t \geq \bar t +1 , \quad n \in \mathbb N^M
\end{align} 
where the $G$ is the solution to the dual equation \eqref{eq: dual} with initial datum 
\[ 
\underline G (0, n ) \leq \mathds{1 }_{ B_{ \overline R }} (n) \quad  \forall n \in \mathbb N^M.  
\]
We can then use the bound \eqref{eq:lower bound subsolution} to deduce that there exists a constant $c_a>0 $ such that 
\begin{align} \label{ineq for F}
\liminf_{ t \to \infty } F(t, n ) \geq   \underline \psi_\infty- \varepsilon + \nu_{\bar R  } ( \overline \psi_\infty - 2 \bar \varepsilon +  \varepsilon - \underline \psi_\infty ) \left( 1-  \dfrac{c \varepsilon_{ \bar R}}{\delta } \left( 1+ c_a \right)\right)  , \quad  \forall t \geq \bar t +1 , \quad n \in \mathbb N^M
\end{align} 
where  $ a, c, \delta  $ are given as in Proposition \ref{prop:subsolution dual} and 
\[ 
\varepsilon_{ \bar R} := \dfrac{1}{a \bar R } \ln \left( 1+ \dfrac{2}{c_a } \right). 
\]

Now notice that by its definition we have that  if $t \geq \overline t +1 $ it holds that 
\[ 
H_3 (t, n ) \geq (\underline \psi_\infty - \varepsilon ) (  G(t, n )  -1 ) \geq - (\underline \psi_\infty - \varepsilon ) \dfrac{c \varepsilon_{ R}}{\delta } \left( 1+ c_a \right) , \quad n \in \mathbb N^M. 
\]

Therefore for every $  n \in \mathbb N^M$ we have that 
\begin{align*} 
\liminf_{ t \to \infty } \psi(t, n ) &  \geq  \underline \psi_\infty- \varepsilon + \nu_{\bar R  } ( \overline \psi_\infty -  2 \bar \varepsilon - \varepsilon - \underline \psi_\infty ) \left( 1 - \frac{\varepsilon_{\bar R } c }{ \delta }\right) - ( \underline  \psi_\infty - \varepsilon ) \dfrac{c \varepsilon_{ R}}{\delta } \left( 1+ c_a \right). 
\end{align*}
Recall that as $R  \to \infty $ we have that $ \varepsilon_R \to 0 $ and recall  that $ \varepsilon >0 $ can be taken arbitrarily small. 
Therefore it is possible to select $ \varepsilon $ and $R $ in such a way that 
\[
 \varepsilon < \nu_{\bar R  } ( \overline \psi_\infty -  2 \bar \varepsilon  - \underline \psi_\infty ) - ( \underline  \psi_\infty - \varepsilon ) \dfrac{c \varepsilon_{ R}}{\delta } \left( 1+ c_a \right) 
\] 
Hence we obtain that
\begin{align*} 
\liminf_{ t \to \infty } \psi(t, n )  >  \underline \psi_\infty. 
\end{align*} 
Due to the definition of $ \underline \psi_\infty $ this  is a contradiction. We deduce that $ \overline \psi_\infty =  \underline  \psi_\infty$. 

 \end{proof}

We are now ready to prove Theorem \ref{thm: convergence to steady state}. 
\begin{proof}[Proof of Theorem \ref{thm: convergence to steady state}]
    The existence of a steady state was proven in Theorem \ref{thm:steady state}. 
    The uniqueness of the steady states is a consequence of the duality formula 
    \begin{equation} \label{duality formula}
    \sum_{ n \in \mathbb N^M } f (T, n ) \psi(0, n ) =     \sum_{ n \in \mathbb N^M } f_0 ( n ) \psi(T, n ) \quad T \geq 0 
    \end{equation}
    where $\psi $ is the solution to \eqref{eq: dual}. 
    Indeed, consider two steady states $f_1 $ and $f_2 $. 
    The duality formula \eqref{duality formula} implies that 
      \[ 
    \sum_{ n \in \mathbb N^M } \left( f_1 ( n ) - f_2(n) \right) \psi(0, n ) =     \sum_{ n \in \mathbb N^M } \left( f_1 ( n ) - f_2(n) \right)  \psi(T, n ) \quad T \geq 0. 
    \]
    If we consider the initial datum $ \psi(0, n )= \delta_{ \overline n } (n) $ for some $ \overline n \in \mathbb N^M$ and take the limit as $T \to \infty $ we deduce that 
     \[ 
    f_1 ( \overline n ) - f_2 ( \overline n ) =  \psi_\infty    \sum_{ n \in \mathbb N^M } \left(  f_1 (  n ) - f_2 (  n ) \right) =0 . 
    \]
    Iterating this argument for every $ n \in \mathbb N^M $ we deduce that the steady state is unique. 

    In order to prove that the time dependent solution $f $ to \eqref{eq:master with sources weak} converges to the unique steady state $\overline f \in \mathcal P(\mathbb N^M ) $ we use again the duality formula. More precisely we consider $ \psi(0, \cdot) \in c_{ 00} ( \mathbb N^M ) $. Then 
       \[ 
    \sum_{ n \in \mathbb N^M } \left( f ( T, n ) - \overline f(n) \right) \psi(0, n ) =   \sum_{ n \in \mathbb N^M } \left( f (0,  n ) - \overline f(n)  \right)  \psi(T, n ), \quad \forall T \geq 0. 
    \]
    Taking the limit as $T \to \infty $ we deduce that 
      \[ 
  \lim_{ T \to \infty }  \sum_{ n \in \mathbb N^M } \left( f ( T, n ) - \overline f(n) \right) \psi(0, n ) = \psi_\infty   \sum_{ n \in \mathbb N^M } \left( f (0,  n ) - \overline f(n)  \right)  =0 . 
    \]
    This implies that $ f(T, \cdot ) \rightharpoonup \overline f $ in the $\ast$-weak  topology as $T \to \infty$. 
\end{proof} 
\bigskip

\section{Models of membrane channels} \label{sec:examples}
Many situations in biological systems can be described using the class of flux solutions considered in the paper. 
In this section we apply the results proven in the paper for general conservative systems endowed with sources and sinks to some models describing the transport of ions and molecules through channels that are located across membranes.
As anticipated in the introduction we consider two models.

In the first model, that we analyse in Subsection \ref{sec:example 1}, we consider the interaction between a channel and some molecules of a specific type. 
The channels opens and closes towards the extracellular space or towards the cytosol at random times. 
The molecules can enter and exit the channel when it is open and are therefore transported down their concentration gradient. 

In the second model, that we study in Subsection \ref{sec:example 2}, we consider the case of co-transport of glucose and Na$^+$ through a channel against a negative glucose gradient. In this case we model the active transport of glucose from the exterior to the interior of the cell membrane. This active transport uses the gradient of concentration of Na$^+$ as a reserve of energy. 
The main feature of this model is that the channel, that is located across the cell membrane,  can switch state only when it is empty or when both a ion Na$^+$ and a molecule of glucose are inside the channel.

\subsection{Model with one type of molecule} \label{sec:example 1}
Assume that a set of molecules, all of the same type, interact with a channel, namely they can pass through the channel. 
We assume that the channel is located across the membrane and separates the region outside the cell (extracellular space) from the region inside the cell (cytosol). 
We denote with $A$ the type of molecules that we consider.
The molecules $A$ can be of two types, they can be outside (type $O$) the membrane or inside the membrane (type $I$). 
We denote with $N$  the number of molecules outside the membrane. 
On the other hand, we denote with $n$  the number of molecules of type $A$ inside the membrane.

In order to go from the outside of the membrane to the inside the molecules of type $A$ have to pass through the channel. 
We assume that the channel can be at two different states. The channel is at state $0 $ if it is empty, i.e. no molecule is in the channel. It is instead at state $1$ if one molecule is in the channel. 
The set of the states at which we can find the molecules and the channel is therefore $ \Omega := \{  0,1 , I , O \} $. 
\begin{figure}[H] 
\centering
\includegraphics[width=0.8\linewidth]{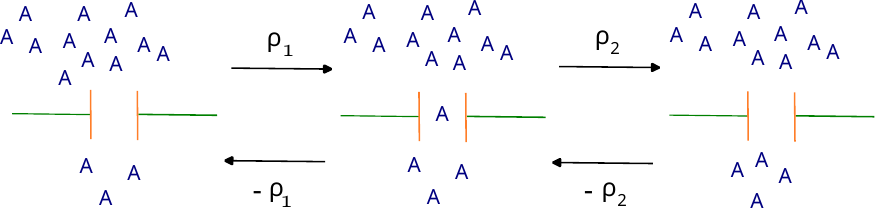}
\caption{ The system can jumps from the state $(n, N , n_X) = ( 3, 6, 1,0) $ to the state $(n, N , n_X) = ( 3, 5, 0,1) $ due to the reaction $ \rho_1$. It can then jump to state $(n, N , n_X) = ( 4, 5, 0,1) $ via the reaction $ \rho_2 $. 
}
\label{fig1}
\end{figure} 
The state of the system is characterized by the triple $ (n,N, n_X) $ where $ n , N \in \mathbb N $ and  $n_X \in \mathbb N^2$ is the number of channels at state $0$ and at state $1$. In particular, we assume that initially we have only one channel. 
The reactions that we consider are the following
\[ 
(A) + (0)  \overset{\rho_1}{\underset{-\rho_1}{\rightleftarrows}} (1), \quad  (1)  \overset{\rho_2}{\underset{-\rho_2}{\rightleftarrows}} (a) + (0). 
\]
Hence we have two reactions vectors 
\[ 
S_{\rho_1 } = \begin{bmatrix}
   -1 \\
     1 \\
     0 \\
     -1 
\end{bmatrix} \ \text{ and } \ 
S_{\rho_2 } = \begin{bmatrix}
     1 \\
     -1 \\
     1 \\
    0 
\end{bmatrix}. 
\]

The stoichiometric matrix associated with the chemical reactions is given by 
\[
\mathbb S_r  =\begin{bmatrix}
 -1 &   1 \\
    1 &  -1 \\
     0 & 1 \\
  -1 &  0 
\end{bmatrix}. 
\] 
Notice that $ \mathbb S_r  w =0$ implies $ w=0$. Hence there are no cycles associated with the system and the detailed balance property holds. 
Moreover the system is conservative, indeed we have that the vector $m =(1,2,1,1) $ is a positive conserved quantity. 

We now add a source and a sink of molecules of type $A$ and $a$, 
\[
\emptyset  \rightleftarrows (I) , \quad \emptyset  \rightleftarrows (O). 
\]
The stoichiometric matrix that includes the sources and sinks is the following
\[
\mathbb S  =\begin{bmatrix}
 -1 &   1 & 0 & 0  \\
    1 &  -1 & 0 & 0 \\
     0 & 1  & 1 & 0 \\
  -1 &  0 & 0 & 1  
\end{bmatrix} 
\] 
Notice that the vector $ w= (1,1,-1,1) $ is a cycle, indeed $ \mathbb S w =0$. 
Therefore Corollary \ref{cor:sources break detailed balance} implies that detailed balance holds if and only if the rates of the sources  $\alpha_k $ and of the sinks $ \beta_k$ satisfy the condition
\[
\log \left( \frac{\alpha_k }{\beta_k }\right) =E(k),\quad  k \in \{ a, A \} 
\]
where $E \in \mathbb R^{4} $ is a solution to
$ \mathbb S_r^T E = v $
where $v \in \mathbb R^4 $ is the vector defined as  $\left( v(k) \right)_{ k =1 }^4=\left( \log \left( \dfrac{K_{-\rho_k} }{ K_{\rho_k} } \right) \right)_{ k =1 }^4 $ where $K_{\rho_k },  K_{-\rho_k }$ are the rates of the reactions $\rho_k  $ and $ - \rho_k $.

We now write the weak form of the master equation \eqref{eq:master with sources weak} associated with the chemical system
\begin{equation} \label{eq:master example}
\dfrac{d }{ dt } \sum_{ n \in \mathbb N }\sum_{ N \in \mathbb N } \sum_{ n_X \in \mathbb N^2 } f(t, n , N , n_X ) \varphi (n , N, n_X ) = \sum_{ n \in \mathbb N }\sum_{ N \in \mathbb N } \sum_{n_X \in \mathbb N^2 } f(t, n , N , n_X) \mathcal L [ \varphi ] (n , N, n_X )
\end{equation}
where $ \varphi \in c_{00} ( \mathbb N^4 ) $ and where  $ \mathcal L = \mathcal L_c + \mathcal L_s $ and  we have that $\mathcal L_c $ is given by 
\begin{align*}
\mathcal L_c [\varphi ] (n,N, n_X  ) =& K_{ \rho_1 } N [  \varphi (n , N -1 , n_X - e_1 + e_2  ) -  \varphi (n , N  , n_X ) ]  \\
& + K_{ -\rho_1 } [ \varphi (n , N +1 , n_X + e_1 - e_2  ) -  \varphi (n , N  , n_X  ) ] \\ 
& + K_{ \rho_2 } n [ \varphi (n+1 , N  , n_X - e_1 + e_2  )  -  \varphi (n , N  , n_X ) ]  \\
& + K_{ - \rho_2 } [\varphi (n-1 , N  , n_X + e_1 - e_2  ) -  \varphi (n , N  , n_X )],
\end{align*}
and $ \mathcal L_s $ is given by 
\begin{align*}
\mathcal L_s [\varphi ] (n,N,  n_X ) =& \alpha_1 [ \varphi(n +1, N , n_X ) - \varphi(n , N ,  n_X  ) ] + \alpha_2  [ \varphi(n , N+1 , n_X ) - \varphi(n , N ,  n_X  ) ] \\
& + \beta_1 n [ \varphi(n -1, N ,  n_X  ) - \varphi(n , N , n_X  ) ] + \beta_2 N [ \varphi(n , N -1 , n_X   ) - \varphi(n , N , n_X   ) ]. 
\end{align*}
We can apply the results proven in this paper to deduce the following theorem.  
As already anticipated above we assume that at time $t=0$ the number of channels is equal to $1$. 
\begin{theorem} \label{thm: example1}
 Assume that $f_0 =\{ f_0(n , N, n_X  ) \}_{ (n , N , n_X   ) \in   \mathbb N^4    } $ is such that $f_0(n , N, n_X  ) \neq 0 $ if and only if $n_X = e_k $ for $ k \in \{ 1, 2 \} $. Moreover assume that 
\[ 
\sum_{ n \in \mathbb N  } \sum_{ N  \in \mathbb N  } \sum_{ n_X \in \mathbb N^2 }   f_0(n, N, e_k ) =1   \ \text{ and } \ \sum_{ n \in \mathbb N  } \sum_{ N  \in \mathbb N  } \sum_{ n_X \in \mathbb N^2 }   (N+n+1)  f_0(n, N, n_X  ) < \infty . 
\]
Then there exists a unique sequence of functions  $ \{ f(t,n,N, n_X ) \}_{ (n , N , n_X  ) \in   \mathbb N^4 \times \{ 0, 1 \}   }$  such that for every $(n,N, n_X)$ the function $t \mapsto f(t,n,N, n-X  )$ is continuously differentiable and satisfy \eqref{eq:master example} for every $ \varphi \in c_{ 00} (\mathbb N^4 ) $. 
Moreover we have that 
\begin{equation} \label{only one channel}
f(t, n , N, n_X  ) \neq 0 \text{  if and only if } n_X = e_k  \text{  for }  k \in \{ 1, 2 \}. 
\end{equation}
Finally, there exists a unique steady state $  \{  \bar f (n, N, n_X ) \}_{ (n , N , n_X  ) \in   \mathbb N^4} $  to equation \eqref{eq:master example} and 
    \begin{equation}
        f(t, (n, N, n_X)  ) \rightharpoonup  \overline f (n, N, n_X) \text { as }  t \to \infty \text{ for every } (n, N, n_X) \in \mathbb N^4.
    \end{equation}
\end{theorem}  
\begin{proof}
The chemical reactions $ S_{\rho_1 }$ and $S_{ \rho_2 } $ are conservative. As a consequence we can apply the results of Theorem \ref{thm:existence and uniqueness of time dep} and Theorem \ref{thm: convergence to steady state}. 
The fact that \eqref{only one channel}  holds is a consequence of the conservation of the number of channels, i.e. of the  fact that for every $ t >0 $ and every test function $ \varphi \in c_{00}(\mathbb N^3)  $ it holds that 
\[ 
   \sum_{ n \in \mathbb N  } \sum_{ N  \in \mathbb N  } \sum_{ n_X \in \mathbb N^2 } \varphi(n, N |n_X|) f(t, n , N, n_X  ) = \sum_{ n \in \mathbb N  } \sum_{ N  \in \mathbb N  } \sum_{ n_X \in \mathbb N^2 } \varphi(n, N |n_X|)  f_0(n , N, n_X  ). 
\]
If we assume now that $ \varphi (n , N , \ell ) = 0 $  if $ \ell =1   $ then we deduce that 
\[ 
   \sum_{ n \in \mathbb N  } \sum_{ N  \in \mathbb N  } \sum_{ \{ n_X \in \mathbb N^2  : | n_X | > 1  \} } \varphi(n, N , |n_X|) f(t, n , N, n_X  ) = \sum_{ n \in \mathbb N  } \sum_{ N  \in \mathbb N  } \sum_{ n_X \in \mathbb N^2 } \varphi(n, N |n_X|)  f_0(n , N, n_X  ) =0 . 
\]
This implies \eqref{only one channel}. 
\end{proof}
Using \eqref{only one channel} the master  equation can be reduced to the following one 
\begin{equation} \label{eq:master example}
\dfrac{d }{ dt } \sum_{ n \in \mathbb N }\sum_{ N \in \mathbb N } \sum_{ k =0  }^2 f(t, n , N , e_k) \varphi (n , N, e_k) = \sum_{ n \in \mathbb N }\sum_{ N \in \mathbb N } \sum_{ k =0 }^2 f(t, n , N , e_k) \mathcal L [ \varphi ] (n , N, e_k)
\end{equation}
where $ \varphi \in c_{00} ( \mathbb N^4 ) $ and where  $ \mathcal L = \mathcal L_c + \mathcal L_s $ and for  $i \neq k, \ i, k \in \{ 0, 1\}  $  we have that $\mathcal L_c $ is given by 
\begin{align*}
\mathcal L_c [\varphi ] (n,N, e_i ) =& K_{ \rho_1 } N [  \varphi (n , N -1 , e_k ) -  \varphi (n , N  , e_i ) ]  + K_{ -\rho_1 } [ \varphi (n , N +1 , e_i ) -  \varphi (n , N  , e_k ) ] \\ 
& + K_{ \rho_2 } n [ \varphi (n+1 , N  , e_i )  -  \varphi (n , N  , e_k ) ]  + K_{ - \rho_2 } [\varphi (n-1 , N  , e_i ) -  \varphi (n , N  , e_k )],
\end{align*}
and $ \mathcal L_s $ is given by 
\begin{align*}
\mathcal L_s [\varphi ] (n,N,  e_k  ) =& \alpha_1 [ \varphi(n +1, N ,  e_k ) - \varphi(n , N ,  e_k ) ] + \alpha_2  [ \varphi(n , N+1 , n_X ) - \varphi(n , N ,  e_k ) ] \\
& + \beta_1 n [ \varphi(n -1, N ,  e_k  ) - \varphi(n , N , e_k ) ] + \beta_2 N [ \varphi(n , N -1 , e_k  ) - \varphi(n , N , e_k  ) ]. 
\end{align*}

Notice that Theorem \ref{thm: example1} states the existence of a stationary distribution to the master equation. Since the system with sources and sinks does not satisfy the detailed balance this stationary solution is a non-equilibrium solution. This means that at the steady state there are non zero fluxes of chemicals crossing the membrane.

\subsection{Model with two types of molecules}  \label{sec:example 2}
The second example we consider is the model of membrane channels described in the book \cite{alberts2022molecular}.
This is a generalization of the example that we study in Section \ref{sec:example 1} to the case in which we have two different types of  molecules in the system.
We denote the two types of molecules $A$ and  $B$. They could be in the exterior of the membrane or in the interior. As indicated above, a possible example is the model of the transport of glucose and Na$^+$  through a channel. 

The interesting feature of the mechanisms of the model that we consider in this section is that it allows to transport molecules against the gradient of concentration. In other words it uses the fact that two molecules must cross the channel at the same time to bring molecules from the region on low concentration to the region of high concentration. 

We now explain the details of the model
We assume that the molecules can go from the outside to the inside of the membrane passing through the channel.   The channel $X$ can be at $8 $ different states. 
\begin{itemize}
    \item State 1:  closed and empty. 
    \item State 2: open towards the exterior and empty. 
    \item  State 3: open towards the exterior and contains a molecule of type $A$. 
    \item   State 4:  open towards the exterior and contains one molecule of type $A$ and one of type $B$. 
    \item  State 5:  $A$ and $B$ inside the channel, which is closed. 
       \item  State 6:  $A$ and $B$ both inside the channel. The channel is open towards the interior. 
       \item State 7: molecule of type $B$  is in the open channel. 
\item  State 8: the channel is empty and open towards the interior.
\end{itemize}

\begin{figure}[H]
\centering
\includegraphics[width=0.8\linewidth]{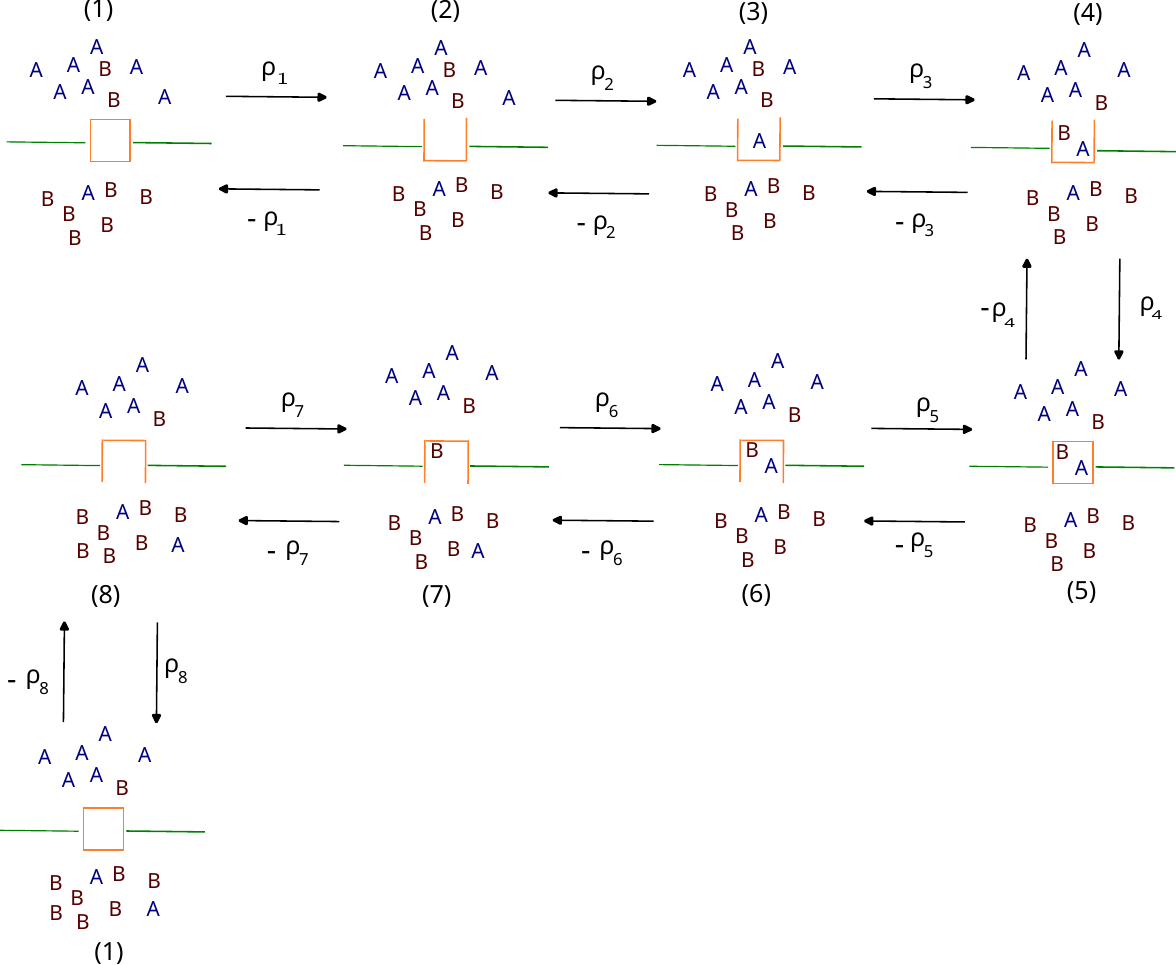} 
\caption{ The channel at the $8 $ different states.  The state of the channel in the figure in the third line and the first state of the first line is the same. However the state of the system in the two cases is different. In the first figure of the first line we have that  $n = (N_A, N_B , n_A , n_B , e_1 ) $ while in the state of the system in the third line is $n = (N_A-1, N_B-1 , n_A+1 , n_B +1, e_1 ) $. 
}
 \label{fig2}
\end{figure} 

The set of the substances is $ \Omega := \{O_A, O_B , I_A, I_B , 1, 2, 3, 4, 5, 6, 7, 8 \} $. Indeed the molecules  $A$ can be of  two types. It is of type $O_A$ if it is outside the cell and of type  $I_A $ if it is inside the cell. 
Similarly, we denote a molecule $B$ that is outside the cell, with $O_B $ and a  molecule $A$ that is inside the cel, with $I_B $.  
We assume that the state of the system is characterized by $n = (N_A, N_B , n_A , n_B , n_X ) $ where $N_A \in\mathbb N $ is the number of molecules of type $A$ outside the membrane, while $n_A \in\mathbb N $ is the number of molecules of type $A$ inside the membrane. Similarly,  $N_B  \in\mathbb N $ is the number of molecules of type $B$ outside the membrane and $n_B \in\mathbb N  $ is the number of molecules of type $B$ inside the membrane. Finally $n_X \in \mathbb N^8 $ is the state of the channel $X$. 

 The reactions that we consider are the following 
\[ 
 (2)\overset{\rho_1}{\underset{-\rho_1}{\rightleftarrows}}  (1)\overset{\rho_2}{\underset{-\rho_2}{\rightleftarrows}}  (8), \quad  (O_A) + (2)  \overset{\rho_3}{\underset{-\rho_3}{\rightleftarrows}} (3), \quad  (3) + (O_B) \overset{\rho_4}{\underset{-\rho_4}{\rightleftarrows}}  (4)  \overset{\rho_5}{\underset{-\rho_5}{\rightleftarrows}} (5)\overset{\rho_6}{\underset{-\rho_6}{\rightleftarrows}}  (6)  , 
 \] 
 \[  (6) \overset{\rho_7}{\underset{-\rho_7}{\rightleftarrows}}  (7) + (I_A) , \quad (7) \overset{\rho_8}{\underset{-\rho_8}{\rightleftarrows}}  (8)+ (I_B). 
\]

It is convenient to write the stoichiometric matrix associated with the chemical reactions
\[
\mathbb S_r  =\begin{bmatrix}
  S_{\rho_k}   
\end{bmatrix}_{k=1}^8 = 
\begin{bmatrix} &1 &-1 & 0 & 0 & 0 &  0 & 0 & 0 \\ 
                               &-1 &0  & -1 & 0 &  0 &  0& 0 & 0 \\ 
                               &0 &0  & 1 & -1 &  0 &  0& 0 & 0 \\
                               &0 &0  & 0 & 1 &  -1 &  0& 0 & 0 \\
                               &0 &0  & 0 & 0 &  1 &  -1& 0 & 0 \\
                               &0 &0  & 0 & 0 &  0 &  1 & -1& 0 \\
                               &0 &0  & 0 & 0 &  0 &  0 & 1 & -1 \\
                               &0 &1  & 0 & 0 &  0 &  0 & 0 & 1 \\
                               &0 &0  &-1 & 0 &  0 &  0 & 0 & 0 \\
                               &0 &0  & 0 & -1&  0 &  0 & 0 & 0 \\
                               &0 &0  & 0 & 0 &  0 &  0 & 1 & 0 \\    
                               &0 &0  & 0 & 0 &  0 &  0 & 0 & 1 \\  
\end{bmatrix}
\]

Notice that $\mathbb S_r w =0 $ implies $w =0$. Hence there are no cycles associate with the system. As a consequence of the Wegscheider criterion detailed balance holds. 

We endow the chemical reactions with  sources and sinks of molecules $A, B ,a, b $, i.e. 
\[ 
\emptyset \rightleftarrows (O_A), \quad \emptyset \rightleftarrows (O_B), \quad \emptyset \rightleftarrows (I_A), \quad \emptyset \rightleftarrows (I_B). 
\]
The stoichiometric matrix of the system with sources and sinks is the following 
\[
\mathbb S  = \begin{bmatrix}  
                               &1 &-1 & 0 & 0 & 0 &  0 & 0 & 0 & 0 & 0 & 0 & 0\\  
                               &-1 &0  & -1 & 0 &  0 &  0& 0 & 0 & 0 & 0 & 0 & 0 \\ 
                               &0 &0  & 1 & -1 &  0 &  0& 0 & 0 & 0 & 0 & 0 & 0 \\
                               &0 &0  & 0 & 1 &  -1 &  0& 0 & 0 & 0 & 0 & 0 & 0\\
                               &0 &0  & 0 & 0 &  1 &  -1& 0 & 0 & 0 & 0  & 0 & 0\\
                               &0 &0  & 0 & 0 &  0 &  1 & -1& 0 & 0 & 0 & 0 & 0\\
                               &0 &0  & 0 & 0 &  0 &  0 & 1 & -1& 0 & 0 & 0 & 0\\
                               &0 &1  & 0 & 0 &  0 &  0 & 0 & 1 & 0 & 0 & 0 & 0\\
                               &0 &0  &-1 & 0 &  0 &  0 & 0 & 0 & 1 & 0 & 0 & 0\\
                               &0 &0  & 0 & -1&  0 &  0 & 0 & 0 & 0 & 1 & 0 & 0\\
                               &0 &0  & 0 & 0 &  0 &  0 & 1 & 0 & 0 & 0 & 1 & 0\\    
                               &0 &0  & 0 & 0 &  0 &  0 & 0 & 1 & 0 & 0 & 0 & 1 \\  
\end{bmatrix}
\]
Notice that the vector $ w=(- 1,-1,1,1,1,1,1,1,1,1,-1,-1)^T \in \mathbb R^{ 12}$ is such that $ \mathbb S w =0$, hence the system has a cycle. 

We can apply Corollary \ref{cor:sources break detailed balance} and deduce that detailed balance holds if and only if the rates of the sources $\alpha_k $ and sinks $ \beta_k$ satisfy the condition
\[
\log \left( \frac{\alpha_k }{\beta_k }\right) =E(k),\quad  k \in \{ a, b, A, B \} 
\]
where $E \in \mathbb R^{8} $ is a solution to
$ \mathbb S_r^T E = v $
where $v \in \mathbb R^8 $ is the vector defined as  $\left( v(k) \right)_{ k =1 }^8=\left( \log \left( \dfrac{K_{-\rho_k} }{ K_{\rho_k} } \right) \right)_{ k =1 }^8 $ where $K_{\rho_k },  K_{-\rho_k }$ are the rates of the reactions $\rho_k  $ and $ - \rho_k $. 

Moreover, notice that the set of chemical reactions associated with $ \mathbb S_r $ are conservative. Indeed the vector $ m^T =(1,1,2,3,3,3,2,1,1,1,1,1)^T$ is a conservation law. 

We now associate to the system with sources and sinks a master equation of the form \eqref{eq:master with sources weak}, i.e. 
\begin{align} \label{eq:master second model of channels}
\sum_{  n \in \mathbb  N^M  } \varphi(n) \partial_t   f (t, n) = \sum_{ n \in \mathbb N ^M }  f (t, n)  \mathcal L_s [\varphi] ( n ) +  \sum_{ n \in \mathbb N ^M }   f (t, n)  \mathcal L_c   [\varphi] (n ) 
\end{align}
where  
\begin{align*}
\mathcal L_c   [\varphi] (n )  : =  \sum_{ k =1}^8  R_{\rho_k}(n) \left[ \varphi(n+ S_{\rho_k}) - \varphi(n)  \right] +   \sum_{ k =1}^8   R_{-\rho_k}(n) \left[ \varphi(n- S_{\rho_k}) - \varphi(n)  \right]  \nonumber \\
\end{align*}
\begin{align*} 
     \mathcal L_s  [\varphi] (n )  :=       \sum_{ j \in \{ A, B , a, b \} }  A_j (n)   \left[ \varphi(n+ e_j ) - \varphi(n)  \right] +    \sum_{ j  \in \{ A, B , a, b \} } B_j (n)   \left[ \varphi(n-  e_j  ) - \varphi(n)  \right].  \nonumber
\end{align*}
\begin{theorem} \label{thm: example1}
 Assume that $f_0 =\{ f_0(n ) \}_{n \in   \mathbb N^{12}    } $ is such that 
\[ 
\sum_{ n \in \mathbb N^{12}  } f_0(n ) =1   \ \text{ and } \ \sum_{ n \in \mathbb N^{12} }   |n|  f_0(n ) < \infty . 
\]
Then there exists a unique sequence of functions  $ \{ f(t,n) \}_{ n  \in   \mathbb N^{12}  }$  such that for every $n \in   \mathbb N^{12}$ the function $t \mapsto f(t,n )$ is continuously differentiable and satisfy \eqref{eq:master second model of channels} for every $ \varphi \in c_{ 00} (\mathbb N^{12} ) $. 
Moreover there exists a unique steady state $  \{  \bar f (n) \}_{ n \in   \mathbb N^{12}} $  to equation \eqref{eq:master example} and 
    \begin{equation}
        f(t, n ) \rightarrow \overline f (n) \text { as }  t \to \infty \text{ for every } n \in \mathbb N^{12}. 
    \end{equation}
\end{theorem}  
\begin{proof}
    Since the chemical reactions $\{ S_{ \rho_k } \}_{ k =1}^8  $ are conservative we can apply  Theorem \ref{thm:existence and uniqueness of time dep} and Theorem \ref{thm: convergence to steady state}. 
\end{proof}

If we assume that initially we have only one channel, more precisely, if we assume that
\[
f_0(n) \neq  0 \text{ if and only if } n = (N_A, n_A, N_B, n_B , n_X) \text{ with } n_X =  e_k  \quad k \in \{ 1, \dots, 8 \}  
\]
then, the conservation of the number of channels implies that the number of channels will be equal to $1$ for all times, i.e. 
\[
f(t,n) \neq  0 \text{ if and only if } n = (N_A, n_A, N_B, n_B , n_X) \text{ with } n_X =  e_k  \quad k \in \{ 1, \dots, 8 \} .
\]

\section{Conclusions and open problems}
In this paper we prove the existence of stationary solution to the master equation of conservative chemical systems with sources and sinks. 
We prove that these stationary solutions are also attractive. 
Typically, these chemical systems do not satisfy the detailed balance property or the complex balance property. The stationary solution is therefore a non-equilibrium solution and cannot be written explicitly as in \eqref{steady state when detailed balance holds}. 
In particular, even when the system is at a stationary state there is a positive probability of having  fluxes of chemicals in the system. 

A relevant question that we will address in a follow-up paper is the  analysis of some limit cases, in which the concentration of some of the substances in the chemical system is very high. 
As in \cite{ball2006asymptotic,remlein2025chemostat}, in this limit case we expect to obtain hybrid deterministic and stochastic dynamics. 
To study these limiting behaviours allows to deduce a precise mathematical definition of reservoirs of chemicals for stochastic systems. 

\bigskip 

\textbf{Acknowledgements.} 
The authors gratefully acknowledge the support of the grant CRC 1720 “Analysis of Criticality: from Complex Phenomena to Models and Estimates” (Project-ID 539309657) of the University of Bonn funded through the Deutsche Forschungsgemeinschaft (DFG, German Research Foundation)  and Germany’s Excellence Strategy-EXC-2047/2–390685813.
The funders had no role in study design, analysis, decision to publish, or preparation of the manuscript.

\bibliographystyle{siam}

\bibliography{References}

\end{document}